\documentclass[11pt]{amsart}
\usepackage[utf8]{inputenc}

\usepackage{mathtools}
\usepackage{amssymb}
\usepackage{tikz}
\usetikzlibrary{cd}
\usetikzlibrary{decorations.pathmorphing}
\usepackage{xcolor}
\usepackage{ragged2e}
\usepackage[style=alphabetic,backend=bibtex, maxbibnames=99]{biblatex}
\renewbibmacro{in:}{}
\usepackage{xcolor}
\definecolor{MyBlue}{RGB}{0,0,255}

\definecolor{MyRed}{RGB}{255,0,0}

\definecolor{MyPink}{RGB}{170,51,106}

\theoremstyle{plain}
\newtheorem{proposition}{Proposition}
\newtheorem{theorem}{Theorem}
\newtheorem{lemma}{Lemma}
\newtheorem{corollary}{Corollary}
\theoremstyle{remark}
\newtheorem{remark}{Remark}
\newtheorem{example}{Example}

\newtheorem{definition}{Definition}

\newcommand{\CC}{\mathbb{C}}
\newcommand{\RR}{\mathbb{R}}

\newcommand{\LL}{\mathcal{L}}

\renewcommand{\aa}{\mathfrak{a}}

\newcommand{\kk}{\mathfrak{k}}

\renewcommand{\ss}{\mathfrak{s}}
\renewcommand{\tt}{\mathfrak{t}}

\newcommand{\uu}{\mathfrak{u}}

\DeclareMathOperator{\Ad}{Ad}
\DeclareMathOperator{\ad}{ad}
\DeclareMathOperator{\Conv}{Conv}

\DeclareMathOperator{\Hess}{Hess}

\DeclareMathOperator{\tr}{tr}

\title[Mabuchi rays on symmetric spaces]{Fibering polarizations and Mabuchi rays on symmetric spaces of compact type}
\author{T.~Baier}
\address{Thomas~Baier\\Center for Mathematical Analysis, Geometry and Dynamical Systems, Instituto Superior T\'ecnico, Lisbon, Portugal.}
\email{thomas.baier@tecnico.ulisboa.pt}

\author{A.~C.~Ferreira}
\address{Ana Cristina Ferreira\\Centro de Matem\'atica, Universidade do Minho, Braga, Portugal.}
\email{anaferreira@math.uminho.pt}

\author{J.~Hilgert}
\address{Joachim~Hilgert\\Department of Mathematics, Paderborn University, Paderborn, Germany.}
\email{joachim.hilgert@upb.de}

\author{J.M.~Mourão}
\address{Jos\'e~M.~Mourão\\Department of Mathematics and Center for Mathematical Analysis, Geometry and Dynamical Systems, Instituto Superior T\'ecnico, Lisbon, Portugal.}
\email{jmourao@tecnico.ulisboa.pt}

\author{J.~P.~Nunes}
\address{Jo\~ao~P.~Nunes\\Department of Mathematics and Center for Mathematical Analysis, Geometry and Dynamical Systems, Instituto Superior T\'ecnico, Lisbon, Portugal.}
\email{jpnunes@tecnico.ulisboa.pt}

\date{\today}

\begin{document}

\begin{abstract}
In this paper, we describe holomorphic quantizations of the cotangent bundle of a symmetric space of compact type $T^*(U/K)\cong U_\CC/K_\CC$, along Mabuchi rays of $U$-invariant K\"ahler structures. At infinite geodesic time, the K\"ahler polarizations converge to a mixed polarization $\mathcal{P}_\infty$. We show how a generalized coherent state transform relates the quantizations along the Mabuchi geodesics such that holomorphic sections converge, as geodesic time goes to infinity, to distributional $\mathcal{P}_\infty$-polarized sections. Unlike in the case of $T^*U$, the gCST mapping from the Hilbert space of vertically polarized sections are not asymptotically unitary due to the appearance of representation dependent factors associated to the isotypical decomposition for the $U$-action. In agreement with the general program outlined in \cite{baier.hilgert.kaya.mourao.nunes:2023}, we also describe how the quantization in the limit polarization $\mathcal{P}_\infty$ is given by the direct sum of the quantizations for all the symplectic reductions relative to the invariant torus action associated to the Hamiltonian action of $U$. 

\end{abstract}

\maketitle

\tableofcontents

\section{Introduction}
\label{sect-introd}

In this paper, we pursue the program outlined in \cite{baier.hilgert.kaya.mourao.nunes:2023} for the geometric quantization of K\"ahler manifolds with an Hamiltonian action of a compact Lie group $U$, by considering  the case of the cotangent bundle of a compact symmetric space,
$$
T^*(U/K) \cong U_\CC / K_\CC.
$$
The main idea is to use Mabuchi rays of K\"ahler structures, generated by the Hamiltonian flows in imaginary time of  Hamiltonian functions which are convex in the moment map, to relate families of holomorphic quantizations to the quantization in a mixed polarization which is attained at infinite Mabuchi geodesic time.

The case of the cotangent bundle $T^*U$ of a Lie group $U$ of compact type  has been studied in \cite{hall:2002, kirwin.mourao.nunes:2013, kirwin.mourao.nunes:2014, baier.hilgert.kaya.mourao.nunes:2023}. The related case of symplectic toric manifolds has been studied in \cite{baier.florentino.mourao.nunes:2011, kirwin.mourao.nunes:2016}
and for recent developments in K\"ahler manifolds with $T$-symmetry see \cite{leung.wang:2022, leung.wang:2023a, leung.wang:2023b}. Applications to the case of flag manifolds are explored in \cite{hamilton.konno:2014} and for more general algebraic varieties in \cite{hamilton.harada.kaveh:2016}

In Section \ref{sect-prelim}, we describe Mabuchi rays of $U$-invariant K\"ahler structures on $T^*(U/K)$, obtained by symplectic reduction from invariant K\"ahler structures on $T^*U$.
These geodesic rays are generated by the Hamiltonian flow in imaginary time of convex functions on $\uu^*$ which are compatible with the symmetric space involution.
At infinite geodesic time along the Mabuchi rays, in Section \ref{sect-mabuchirays}, we obtain the mixed polarization $\mathcal{P}_\infty$ on the open dense subset of regular values of the moment map for the right $K$-action on $T^*U$, $T^*(U/K)_\mathrm{reg}$. In Section \ref{sect-extendedquantumbundle}, we show how a generalized coherent state transform relates the K\"ahler quantizations of $T^*(U/K)$ along the Mabuchi rays such that, in the limit of infinite geodesic time, the elements of natural basis of holomorphic sections, given by the isotypical decomposition with respect to the $U$-action, converge to distributional polarized sections for $\mathcal{P}_\infty.$ In Section \ref{subsection-hinfinity}, we show that independently of the Mabuchi ray which one follows, connecting the vertical, or Schr\"odinger, polarization to $\mathcal{P}_\infty$, there is a well-defined limit for the inner product structures along the family of K\"ahler polarizations. This leads to a natural definition of Hilbert space structure on the space of $\mathcal{P}_\infty$-polarized section $\mathcal{H}_\infty.$ As described in Sections  \ref{subsection-hinfinity} and \ref{subsection_fouriertransform}, in contrast to the case of $T^*U$ studied in \cite{baier.hilgert.kaya.mourao.nunes:2023}, the resulting $U$-equivariant isomorphism $\mathcal{H}_\mathrm{Sch} \to \mathcal{H}_\infty$ defined by the gCST is not asymptotically unitary (even for Mabuchi rays generated by Hamiltonians quadratic in the moment map) and unitarity is achieved only by including representation-dependent correcting factors for each isotypical component. Finally, in Section \ref{subsection_interpretation}, we give a quantum-geometrical interpretation of $\mathcal{P}_\infty$, in the line of the general program described in \cite{baier.hilgert.kaya.mourao.nunes:2023}, where we relate the K\"ahler quantization of $T^*(U/K)$ with a direct sum over the quantizations of the coisotropic reductions for $\mathcal{P}_\infty$ or, equivalently, over the quantizations of the symplectic reductions for the invariant moment map.

{}

\section{Preliminaries}
\label{sect-prelim}

\subsection{Basic definitions}

Let $U$ be a compact simply connected Lie group and $$\sigma:U\to U$$ an involutive automorphism 
such that $U/K$ is a symmetric space of compact type, where $K$ is a closed subgroup of $U$ and a relatively open subgroup of the set $U^\sigma$ of fixed points\footnote{Since we are taking $U$ to be simply connected, from \cite[Lemma II.6.2]{Takeuchi94} we also have  that $U^\sigma$ and hence $K$ is connected.}. One has for the Lie algebra {$\uu$} of $U$, and for its dual $\uu^*$, orthogonal decompositions with respect to a fixed $\Ad_U$-invariant inner product, $\langle\cdot, \cdot\rangle_\uu$, on $\uu$,
$$
\uu = \kk \oplus \ss, \,\, \uu^* = \kk^* \oplus \ss^*,
$$
where, for the derived automorphism of $\uu$,  also denoted by $\sigma$,
$$
{\sigma_{\vert{\kk}} = \mathrm{Id}_\kk}, \, \sigma_{\vert{\ss}} = -\mathrm{Id}_\ss,
$$
with $[\kk,\ss] \subset \ss$ and $[\ss,\ss] \subset \kk$.

The cotangent bundle $T^*(U/K)$ has the structure of a {homogeneous} vector bundle associated to the principal $K$-bundle $U\to U/K$,
$$T^*(U/K) =U\times_K \ss^*.
$$

\subsection{Invariant K\"ahler structures on $T^*(U/K)$}

Here, we will recall standard facts about the symplectic geometry of the cotangent bundles $T^\ast (U)$ and $T^\ast (U/K)$, realized as the symplectic quotient
\[
  T^\ast \left(U/K\right) = \left(T^\ast U\right) /\!\!/ K \, .
\]

Recall that there is a Hamiltonian right action of $U$ on $T^*U$ with moment map $\mu^{(R)}$. The moment map of the corresponding right action of the subgroup $K\subset U$ is
\begin{equation}
\label{ee-muk}
 \mu^K = \pi_{\kk^\ast} \circ \mu^{(R)} ,
 \qquad \pi_{\kk^\ast}: \uu^\ast \to \kk^\ast .
\end{equation}
Using the trivialization $T^*(U)\cong U\times \uu^\ast $ one finds that
\begin{equation}
 \label{e-quot-const}
 \left(\mu^K \right)^{-1}(0) = U \times \ss^\ast
 \quad \text{ and } \quad 
 \left(T^\ast U\right) /\!\!/ K =
 \left( \mu^K \right)^{-1}(0) / K =
 U \times_K \ss^\ast = T^*(U/K).
\end{equation}

The moment map of the left $U$-action on $T^\ast \left(U/K\right)$, which we will denote ${\mu}$, descends from $T^*U$ through this quotient,
\begin{equation}
\label{e-moment-symm-space}
 \begin{tikzcd}
  U \arrow[r, phantom, "{\circlearrowright}"] & T^\ast \left(U/K\right) \arrow[r, "{{\mu}}"] & \uu^\ast
 \end{tikzcd}, \qquad
{\mu}([x,\xi]) = \Ad^\ast_x \xi ,
\end{equation}
where $[x, \xi]$ denotes 
the $K$-orbit through 
the point $(x,\xi)$ in
$U \times \ss^*$.

If we equip $T^*U$ with a $U\times U$-invariant K\"ahler structure, then this symplectic quotient becomes a K\"ahler quotient so that $T^*(U/K)$ inherits an $U$-invariant K\"ahler structure with respect to the reduced symplectic structure.
We will now describe an infinite-dimensional space of $U$-invariant K\"ahler structures on $T^* (U/K)$ obtained by K\"ahler reduction from the $(U \times U)$-invariant K\"ahler structures on $T^*(U)$, which are described in Section 2.3 of \cite{baier.hilgert.kaya.mourao.nunes:2023} and in \cite{neeb:2000a, neeb:2000b, kirwin.mourao.nunes:2013}. Given a uniformly convex invariant function on $\uu^*$,
$$
g \in  \Conv^\infty_{\rm unif}(\uu^\ast)^{\mathrm{Ad}_U},
$$
one has the Legendre transform $\mathcal{L}_g:
T^*U \cong U\times \uu^*\to U_\CC$ given by
$$
\mathcal{L}_g (x,\xi) = x e^{id_\xi g}.
$$
The pull-back of the canonical complex structure on $U_\CC$ by $\mathcal{L}_g$ and the canonical symplectic structure on $T^*U$ define a $U\times U$-invariant K\"ahler structure on $T^*U$, $I_g$, with global K\"ahler potential
$$
\kappa_g(x,\xi) = \langle \xi, d_\xi g\rangle - g(\xi).
$$
{}
\begin{lemma} \label{lemma.1}
For any uniformly convex $g \in  \Conv^\infty_{\rm unif}(\uu^\ast)^{\mathrm{Ad}_U}$, the corresponding K\"ahler potential $\kappa_g: U_{\CC}\to \RR$ is exhausting.
\end{lemma}
\begin{proof}
We need to show that the component in $\uu^\ast$ of the preimage of an interval $]-\infty,\lambda[$ under $\kappa_g \circ {\LL}_g$ is bounded, which follows if we show that
 $ \kappa_g \circ {\LL}_g $ is not bounded above on unbounded subsets of  $\uu^\ast$.
As $g$ is uniformly convex on $\uu^\ast$, its Legendre transform maps unbounded subsets to unbounded subsets in $\uu$, but on these $\kappa_g$ is unbounded as it is supercoercive by a proposition of Moreau and Rockafellar \cite[Prop. 3.5.4]{borwein.vanderwerff:2010}.
\end{proof}

\begin{remark}
 As we will see in Theorem \ref{th-4}, the exhaustion property of the global K\"ahler potential $\kappa_g$ is needed to show that all points in $U_\CC$ are  stable with respect to the $K_\CC$-action in the sense that all $K_\CC$-orbits intersect the zero level set of the $K$-moment map. 
\end{remark}
{}
\begin{remark}\label{gissigmacompatible}Let $\sigma$ also denote the involution on $\uu^*$ induced from $\sigma:\uu\to \uu,$ so that 
$$
{\sigma_{\vert{\kk^*}} = \mathrm{Id}_{\kk^*}}, \, \sigma_{\vert{\ss^*}} = -\mathrm{Id}_{\ss^*}.
$$
In this paper, we will always assume that the symplectic potential $g$ (and also $h$ and $g_t = g+th, t>0,$ to appear below), are compatible with the involution $\sigma$, that is we assume that 
$$g\circ \sigma = g.$$ This is not a major restriction since symplectic potentials built from the even degree Casimirs of $\uu$ will satisfy this property, to begin with the fundamental example given by $g= \frac12 \vert\vert\xi\vert\vert^2.$ 
In addition to the properties listed in Proposition 2.7 of \cite{baier.hilgert.kaya.mourao.nunes:2023}, we will therefore also have that:
$$d_\xi g= d_{\sigma(\xi)}g \circ \sigma = \sigma( d_{\sigma(\xi)}g),
$$
so that, in particular, for $\xi\in\ss^*$ we get 
$$
(d_\xi g)_{\vert_{\ss^*}} =  - (d_{-\xi}g)_{\vert_{\ss^*}},
$$
and
$$
(d_\xi g)_{\vert_{\kk^*}} =  (d_{-\xi}g)_{\vert_{\kk^*}}.
$$
For the Hessian of $g$ we also obtain (see (c) in Proposition 2.7 of \cite{baier.hilgert.kaya.mourao.nunes:2023}), for $\xi\in \ss^*$,
$$
\mathrm{Hess}_g(\xi)_{\vert_{\ss^*}} =  \mathrm{Hess}_g(-\xi)_{\vert_{\ss^*}}
$$
These identities will be used below. 
\end{remark}

\begin{remark}\label{legendreok}
Note that, from Lemma 3.1 in \cite{kirwin.mourao.nunes:2013}, one has that the map 
$\uu^* \ni \xi \mapsto d_\xi g \in \uu$ is a diffeomorphism. The compatibility condition between $g$ and $\sigma$, imposed in Remark \ref{gissigmacompatible}, then gives that for $\xi\in\kk^*$,
$$
d_\xi g = \sigma d_\xi g,
$$
which implies that $d_\xi g\in \kk$ for $\xi\in \kk^*.$ The fact that $g$ is compatible with $\sigma$ also implies that the K\"ahler potential $\kappa_g$ is compatible with $\sigma$ so that the inverse Legendre transform, which is the Legendre transform  with respect to $\kappa_g$, maps $\kk$ bijectively onto $\kk^*$. It follows that, since $\uu = \kk \oplus \ss$, 
$$
\xi \in \ss^* \Leftrightarrow d_\xi g \in \ss.
$$
\end{remark}

{}

\begin{theorem}\label{th-4}
Symplectic reduction provides us with a map 
\begin{equation}\label{UxN-invariant_Is}
 \Conv^\infty_{\rm unif}(\uu^\ast)^{\Ad^\ast_U} \overset{\widehat{I}}{\longrightarrow} \mathcal{J}(T^\ast (U/K),\omega_{\rm std})^{U} ,
 \qquad
 g \mapsto \widehat{I}_g,
\end{equation}
where $\mathcal{J}(T^\ast (U/K),\omega_{\rm std})^{U}$ is the space of $U$-invariant complex structures on $T^\ast(U/K)$ compatible with the standard symplectic form $\omega_{\rm std}$.
Dually, these K\"ahler structures are described by a map
\begin{equation}\label{UxN-invariant_metrics}
 \Conv^\infty_{\rm unif}(\uu^\ast)^{\Ad^\ast_U} \overset{\widehat{\omega}}{\longrightarrow} \mathcal{K}(U_{\CC}/K_{\CC})^{U} , \qquad g \mapsto \widehat{\omega}_g,
\end{equation}
to the space of $U$-invariant K\"ahler forms on $U_{\CC}/K_{\CC}$.

Explicitly, these maps are defined by pulling back the relevant structures along the map descending from the Legendre transform ${\LL}_g$, which we still denote by ${\LL}_g$, through the symplectic quotient,
\begin{equation}\label{legendre_on_reduction}
    \begin{tikzcd}[row sep=small]
      U \times \ss^\ast \arrow[r, "{\LL_g\vert_{U \times \ss^\ast}}"] \arrow[d] & U\times \uu \arrow[r] & U_{\CC} \arrow[d] \\
      U \times_K \ss^\ast \arrow[rr, "{{\LL}_g}"] & & U_{\CC}/K_{\CC} \\[-3ex]
     {[x,\xi]} \arrow[rr, mapsto] & & x e^{i d_\xi g} K_{\CC}
    \end{tikzcd} .
\end{equation}
In particular, a K\"ahler potential $\widehat{\kappa}_g$ for $\widehat{\omega}_g$ is given by the restriction of $\kappa_g$ to the \emph{Kempf--Ness set} $\left(\mu^{K}\right)^{-1}(0) = U \times \ss^\ast$ (see (\ref{ee-muk})),
\begin{equation}\label{ee-9cc}
  \widehat{\kappa}_g \circ {\LL}_g ([x,\xi]) =
  \langle \xi, d_\xi g \rangle - g(\xi) . 
\end{equation}
\end{theorem}

\begin{proof}
Since, by Lemma \ref{lemma.1}, $\kappa_g$ is a global $(U\times U)$-invariant strictly plurisubharmonic exhaustion function on $U_\CC$, it follows from \cite[Lemma 2.4.2]{heinzner.huckleberry:1999} that all points in $U_\CC$ are
$\mu^K \circ (\LL_g)^{-1}$-stable.
Considering the restriction of $\LL_{g}$ 
to $(\mu^K)^{-1}(0)= U \times \ss^*$ as in (\ref{legendre_on_reduction}),
we know that both these maps are diffeomorphisms 
onto their images.
From the equivariance of the Legendre transform (see Proposition 2.7 in \cite{baier.hilgert.kaya.mourao.nunes:2023}), we know that the map $\LL_{{g}}|_{U \times \ss^*}$ descends to the $K$-quotients.
The map $\LL_{g}$ in the lower arrow of (\ref{legendre_on_reduction}) then corresponds  to the map $i_x$ in \cite[Lemma 2.4.3]{heinzner.huckleberry:1999} and hence defines a diffeomorphism between the symplectic and the Hilbert quotients. Thus, it  defines the K\"ahler quotient structure on these quotients.

The form of the K\"ahler potential in (\ref{ee-9cc}) follows from the fact that the K\"ahler potential of a K\"ahler reduction is obtained by restricting the K\"ahler potential to the Kempf--Ness set, cf. \cite[Proposition 2.4.6]{heinzner.huckleberry:1999} or \cite[Theorem 7]{biquard.gauduchon}. 
 \end{proof}
 \begin{remark}
  We note that the fact that the map (\ref{legendre_on_reduction}) is a diffeomorphism is not at all easy to see from Lie theoretic arguments, not even in the case of $g(\xi) = |\xi|^2/2$, for which
  $d_\xi g = \xi$.
 \end{remark}
We denote by $\widetilde{\mathcal K}(U/K)$
the space of K\"ahler structures
on $T^*(U/K)$ corresponding to the image of 
the map $\widehat I$ in 
(\ref{UxN-invariant_Is})
$$
\widetilde{\mathcal K}(U/K) := \left\{
(\omega_{\rm std} , \widehat{I}_g) \mid  g \in \Conv^\infty_{\rm unif}(\uu^\ast)^{\Ad^\ast_U}
\right\} \, .$$

\subsection{Restricted roots, Satake diagrams}
\label{subsection_restrictedroots}

Let $U/K$ be an irreducible symmetric space of compact type and recall that 
$T^*(U/K) \cong U \times_K \ss^*$, with $U$-moment map
given by (\ref{e-moment-symm-space}).
We will assume that the Cartan
subalgebra $\tt$ of $\uu$ is chosen $\sigma$-invariant such that
$$i \aa := \tt \cap \ss \subset \ss$$
is a maximal Abelian subspace of
$\ss$. 

We will now follow \cite{araki:1962, warner:1972, helgason:1984, helgason:2001} 
to describe aspects of the
geometry of the symmetric space $U/K$ through the
properties of the symmetric
pair
$(\uu_\CC,\kk_\CC)$.
Let, again, $\sigma$ denote the antilinear involution 
extended to $\uu_\CC$ by antilinearity
from the following linear involution on $\uu$,
$$
\sigma(X) = \left\{
\begin{array}{rl}
X     &  \hbox{if}  \, \,  X \in \kk\\
-X     & \hbox{if}  \, \,  X \in \ss  \, .
\end{array}
\right.
$$
{}
Note that $\sigma_{\vert_{\aa}}= \mathrm{Id}_{\aa}.$
The antilinear involution $\sigma$
defines an antilinear 
involution on $\uu^*_\CC$,
which leaves the root system $\Phi$
of $(\uu_\CC, \tt_\CC)$ invariant, 
and  defines an involutive isometry of $\Phi$
\cite[{\bf 1.3}]{araki:1962}.
\begin{eqnarray}
\nonumber
\label{e-anti-lin-inv}
\uu_\CC^* \, \ni  \xi & \mapsto & \xi^\sigma   \in \uu_\CC^* \\
\xi^\sigma(X) & := & \overline{\xi(\sigma(X))} \, , \quad \forall X \in \uu_\CC  \, . 
\end{eqnarray}
Let $E_\alpha, \alpha\in \Phi$, be a choice of root vectors such that $[E_\alpha, E_{-\alpha}]=H_\alpha$, where $H_\alpha$ denotes the coroot vector corresponding to $\alpha$, that is $\beta(H_\alpha) = \langle \alpha,\beta\rangle, \beta\in \Phi$, where $\langle\cdot, \cdot\rangle$ denotes the inner product on $i\tt$ obtained from the fixed invariant inner product on $\uu$ and extended to $\uu_\CC$.
If $\uu_\alpha := \CC E_\alpha$ denotes the 
root space generated by the vector $E_\alpha$,
then
$$
\sigma (E_\alpha)  \in 
\uu_{\alpha^\sigma}  \, .
$$

The pair $(\Phi, \sigma)$
is called a {\it $\sigma$-root system}.
The subset
$$
\Phi_0 := \left\{\alpha \in 
\Phi \mid
\alpha^\sigma = -\alpha \, 
\right\} = \left\{\alpha \in 
\Phi \mid   \alpha_{|_\aa} = 0  \right\}
\, 
$$
is a root subsystem of
$\Phi$.

The order 
of $i \tt^*$ 
with positive cone 
${\rm Cone}(\Lambda_+)$, where $\Lambda_+$ is the set of dominant integral weights induced by the choice of positive roots, $\Phi_+ \subset \Phi$,
is called
a {\it $\sigma$-order} if 
\cite[{\bf 2.8}]{araki:1962}, 
\cite[p. 23]{warner:1972}
\begin{equation}
\label{e-sigma-order}
\alpha \in \left(\Phi \setminus \Phi_0\right) \cap \Phi_+ 
\Rightarrow   \alpha^\sigma \in 
\Phi_+ \, .
\end{equation}
If $\Delta$ is a system of 
simple roots for a $\sigma$-order,
then $\Delta_0 := \Delta \cap \Phi_0$ is a system of simple roots for $\Phi_0$. Let $r:={\rm rank}(\Phi), \, 
r_0:={\rm rank}(\Phi_0)$
and
$$
\Delta = \left\{\alpha_1, \dots, 
\alpha_{r-r_0},
\alpha_{r-r_0+1}, \dots 
, \alpha_r\right\}
$$
with
$$
\Delta_0 = \left\{\alpha_{r-r_0+1}, \dots, 
\alpha_{r}\right\}   \, .
$$
Then $\sigma$ induces a permutation of 
$\Delta \setminus \Delta_0$ as
there is an involutive  permutation $\tilde \sigma$ of
$\{1, \dots , r-r_0\}$ such that \cite[Lemma 1.1.3.2]{warner:1972}
\begin{equation}
\label{e-perm-sigma-tilda}
\alpha_j^{\sigma} =
\alpha_{\tilde \sigma(j)} +
\sum_{k=r-r_0+1}^r \, c_k^{(j)}
\, \alpha_k    \, , \quad j = 1, \dots , r-r_0 \, ,
\end{equation}
with coefficients
$c_k^{(j)} \geq 0$.
Then the restriction of
roots to $\aa$,
$$
\Sigma : = \left(
\Phi \setminus \Phi_0\right)|_{\aa} \subset \aa^*  \, ,
$$
is a root system, called the
{\it restricted root system} 
of the $\sigma$-system 
$(\Phi, \sigma)$,
with system of simple
roots,  $\Delta^-$ given by
\begin{equation}
  \label{e-simple-restricted-roots}
\Delta^- :=
\left\{\beta_1, \dots , \beta_{l}  \right\} :=
\left(
\Delta \setminus \Delta_0\right)|_{\aa}   \, , \, 
  \end{equation}
where $l := \dim(\aa)= {\rm rank}(U/K)$. 
Note that, if $\alpha\in \Phi\setminus\Phi_0$ then, since $\alpha\in i\tt^* = i\, \mathrm{Hom_\RR}(\tt,\RR)$, one obtains  
$\alpha^\sigma_{\vert_{\mathfrak a}} = \alpha_{\vert_{\mathfrak a}}$. 

We see from
(\ref{e-perm-sigma-tilda})
that the simple roots
$\alpha_j$ and $\alpha_{\tilde \sigma(j)}$ map to the same
restricted root $\beta_j$.
These are the only relations 
for the restriction of simple roots from $\Delta \setminus \Delta_0$ to $\aa$ and the
restricted root systems can be obtained from the Dynkin diagram 
of $\Delta$ by a {\it Satake diagram}:
Simple roots from $\Delta_0$ are
represented by black circles and
roots from $\Delta \setminus \Delta_0$ are represented by white circles. White circles related by the permutation $\tilde \sigma$
are linked by a curved arrow. One has, therefore, 
\begin{equation}
    \label{e-role-of-arrows}
    l = r - n_b -n_a,
\end{equation}
where $n_b$ is the number of black circles and $n_a$ the number of arrows in the Satake diagram for $U/K$.
The classification of irreducible
symmetric spaces  
is then given by
the possible Satake diagrams. Those corresponding to
simple $\uu$ are listed by
the tables in \cite[p. 32, 33]{araki:1962} and
\cite[p. 30--32]{warner:1972}.

It follows from \cite[Lemma VI.3.6]{helgason:2001}
that $\ss_\CC$ has the following decomposition
\begin{equation}
\label{e-rest-root-space-dec}
 \ss_\CC = \aa_\CC + \sum_{\alpha \in (\Phi \setminus \Phi_0)\cap \Phi_+} \, 
 \CC \left(E_\alpha - \theta   E_{\alpha}\right)
    \, , \\
\end{equation}
where 
$\theta$ denotes the 
$\CC$--linear involution 
on $\uu_\CC$ that on 
$\uu$ coincides with $\sigma$ so that
$\theta|_{\kk_\CC} = {\rm Id}_{\kk_\CC}$
and 
$\theta|_{\ss_\CC} = - {\rm Id}_{\ss_\CC}$. 

\subsection{Spherical representations}
\label{subsection_sphericalreps}

Recall that we are assuming  that $U$ is simply-connected.\footnote{Just as in \cite{baier.hilgert.kaya.mourao.nunes:2023}, this condition makes the presentation simpler but we expect the generalization to other cases to be straightforward. See Remark \ref{remark-unotsimplyc}.} (Note that since $U$ is compact one has a splitting $\uu = [\uu,\uu]\oplus \mathfrak{b}$ where $\mathfrak{b}$ is abelian; since $U$ is simply-connected one has $\mathfrak{b}=\left\{0\right\}$ so that $U$ is  semisimple.)
We have the isomorphism 
$$
 L^2(U/K) \cong L^2(U)^K  \, ,
$$
where $L^2(U)^K$
denotes the subspace of
right $K$-invariant functions. 
Let $\hat U$ denote the set of equivalence classes of irreducible representations $V_\lambda$ of $U$, labelled by highest weight $\lambda$. Recall the Peter-Weyl theorem giving an orthogonal decomposition
$$
L^2(U) \cong \widehat \bigoplus_{\lambda \in \hat U} \mathrm{End}(V_\lambda),
$$
where the hat over the sum denotes the norm completion. The summand 
$$
\mathrm{End} (V_\lambda) \cong V_\lambda\otimes V_\lambda^*
$$
is realized by 
$$
f_{\lambda, v\otimes w^*} (x)= \tr (\pi_\lambda(x) v\otimes w^*) \in L^2(U), \, \, \lambda \in \hat U, v\in V_\lambda, w^*\in V_\lambda^*,
$$
where $\pi_\lambda(x)\in \mathrm{End}(V_\lambda)$ 
denotes the representative of $x\in U$.
Then $L^2(U/K)$ is  given by the terms which are invariant under the right $K$-action. From, Theorem 4.1 in Chapter 5 in \cite{helgason:1984}, each $V_\lambda$ in $\hat U$ contains at most a one-dimensional subspace of so-called $K$-spherical vectors which are invariant under $K$. The corresponding representations are called $K$-spherical representations and we denote them by $\hat U_K$. 

If $\lambda\in \hat U_K$ let $v_\lambda^K\in V_\lambda$ be a non-trivial $K$-spherical vector and let $V_\lambda^K=\CC \cdot v_\lambda^K$ be the one-dimensional subspace of $K$-spherical vectors. Then, we obtain an orthogonal decomposition
$$
L^2(U/K) \cong \widehat \bigoplus_{\lambda\in \hat U_K}
\mathrm{Hom}(V_\lambda,V_\lambda^K),
$$
where the summand $\mathrm{Hom}(V_\lambda,V_\lambda^K)\cong V_\lambda^K\otimes V_\lambda^*$ is realized by the right $K$-invariant functions
$$
f_{\lambda, v_\lambda^K\otimes v^*}\in L^2(U), \, v^*\in V_\lambda^*.
$$
(Note that the choice of spherical vector $v_\lambda^K$ is unique up to a multiplicative constant $c$ and that $f_{\lambda, cv_\lambda^K\otimes v^*}= c f_{\lambda, v_\lambda^K\otimes v^*}$.)

From Theorem 4.1 in Chapter 5 in \cite{helgason:1984} we also obtain the explicit characterization of the highest weights for the $K$-spherical representations. If $\tt_{\mathbb Z}^*$ denotes the weight lattice for $U$, we have that a highest weight 
$\lambda\in \tt_{\mathbb Z}^*$ is the highest weight of a $K$-spherical representation if and only if 
$$
\lambda(i(\tt \cap \kk)) =0.
$$
We denote by $\Lambda_+^K$ the set of highest weights of $K$-spherical representations. In particular, the cone generated by $\Lambda^K_+$ is
$$
\mathrm{Cone}(\Lambda^K_+) = \aa_+^*,
$$
where $\aa_+^* := i(\tt^*_+ \cap \ss^*)$ and $\tt^*_+$ is the closed positive Weyl chamber associated to the choice of positive roots $\Phi_+\subset \Phi$, defined by
$$
\tt^*_+ = \left\{ \xi\in\tt^*\,: \, \langle\alpha, i\xi\rangle\geq 0, \alpha\in \Phi_+\right\}
$$
(Note that 
$\aa^* = \mathrm{Hom}_\RR (\aa, \RR) = \mathrm{Hom}_\RR (\tt\cap \ss, i\RR) = i(\tt^*\cap \ss^*).$)

\begin{remark}\label{remark-unotsimplyc}
    If $U$ is not simply-connected, the above needs to be adapted since not all representations of $\uu$ will integrate to representations of $U$. 
\end{remark}

\begin{remark}\label{rmk-legendreok2}
 In particular, from Remark \ref{legendreok}, since the Legendre transform maps the  positive Weyl chamber in $\uu^*$ to its dual (see Lemma 4.7 in \cite{baier.hilgert.kaya.mourao.nunes:2023}), $d_{\xi_+}g\in -i\aa_+=\tt_+\cap \ss$ if and only if $\xi_+ \in -i\aa_+^*=\tt_+^*\cap \ss^*.$    
\end{remark}

{}

\subsection{The invariant moment map and invariant torus action}

Recall the 
sweeping map $s:\uu^*\to \tt^*_+$, given by conjugation to the positive Weyl chamber. The invariant moment map is
\begin{equation}\label{eq:mu_inv}
   \mu_\mathrm{inv} := s \circ \mu, 
\end{equation}
whose image defines the Kirwan polytope. The $\mu$-regular stratum
$T^*(U/K)_\mathrm{reg}$ is the pre-image under $\mu_\mathrm{inv}$ of the relative interior of the top-dimensional face of the Kirwan polytope which is called the principal stratum. Along $T^*(U/K)_\mathrm{reg}$, as we will recall later, $\mu_\mathrm{inv}$ is the moment map for a smooth effective Hamiltonian torus action by a torus $T_\mathrm{inv}$ which will play a crucial role in our analysis. The torus $T_\mathrm{inv}$ is a quotient by a discrete subgroup of the torus $\tilde T_\mathrm{inv}$ whose Lie algebra is determined by the principal stratum. Let $\tau_{U/K}$ denote the principal stratum of $T^*(U/K)\cong U\times_K \ss^*.$

\begin{proposition}
A point $[x, \xi] \in U\times_K \ss^*$ is 
$\mu$--regular if and only if
\begin{equation}
\label{e-muhat-regular}
\langle s(\xi) , \alpha\rangle \neq 0
\, , \quad \forall \alpha \in \Phi \setminus \Phi_0  \, .
\end{equation}
\end{proposition}
\begin{proof}
The point $[x,\xi]$ being
{\it not} $\mu$-regular means that
$s(\mu([x, \xi]))
= s(\xi)$ is in the 
(relative) boundary of $\tau_{U/K}$. 
From the discussion in Section \ref{subsection_restrictedroots},
we see that this
happens
if and only if
$\langle s(\xi), \beta \rangle =0$, for some $\beta \in \Sigma$, which,
since $s(\xi) \in i \aa_+^*= (\tt^*_+\cap\ss^*)$, is equivalent
to
$\langle s(\xi), \alpha \rangle =0$
for some root
$\alpha \in \Phi\setminus \Phi_0$.
\end{proof}

Concerning the momentum
set of $T^*(U/K)$
one has the following.
\begin{proposition}
\label{e-mom-set}
The Kirwan polytope of $T^* (U/K)$
is $-i$ times the cone generated by dominant $K$-spherical weights
\begin{equation}
\label{e-cone}
\mu(T^*(U/K)) \cap \tt^*_+ 
%= -i \, \, {\rm Cone}(\Lambda^K_+)
= -i \, \aa^*_+\, . \end{equation}
\end{proposition}

\begin{remark}    
The principal stratum
$\tau_{U/K}$ is the
unique stratum with 
nonempty intersection with the
relative interior of this cone. Therefore, the dimension of $\tau_{U/K}$ is equal to the rank of $U/K$. 
\end{remark}
%Its codimension in $\tt^*$ is equal to
%the number of simple roots
%which restrict to zero on $\aa$
%(= number of black circles, $n_b$,  in the Satake diagram), 
%$$
%\dim(\tt^*) - {\rm dim}(\tau_{U/K}) = 
%n_b \, . 
%$$
%momentum cone $\aa_+$ ($=$ rank of the symmetric space), 
%i.e. the dimension of the kernel $C$ in $\hat T_\mu = T_\mu/C$,
%$\cc_\tau$ of $i^*$ in (\ref{e-ex-seq-lie-algeb})

\begin{remark}
Equation (\ref{e-cone}) implies that two different roots, $\alpha, \gamma
\in \Phi \setminus \Phi_0$, 
that are equal when restricted to $\aa$ correspond to the same component of $\mu_\mathrm{inv}$ with respect to a basis of $i\aa^*$ given by restricted roots.
(The components of $\mu_\mathrm{inv}$ are also called Guillemin-Sternberg action 
coordinates.) 
%That is, $Lie (T_\mathrm{inv}) = {\mathfrak t}_\mu^* /\hskip-%0.3em\sim$, where $\hat \alpha \sim \hat \gamma$ if  %$\hat \alpha_{\vert_{\mathfrak a}}= \hat %\gamma_{\vert_{\mathfrak a}}$, where
%$$
%\hat \alpha = \hat \gamma  \, 
%\in C^0(T^*(U/K)).
%$$
\end{remark}

\begin{proof}
From the form of the moment map (\ref{e-moment-symm-space})
it follows  that the image of the moment map contains the cone 
\begin{equation}
\label{e-aplus-cone}
\nonumber
 -i\aa^*_+ = \tt_+^*  \cap \ss^*   \, .
\end{equation}
On the other hand 
\cite[Theorem III.4.14]{helgason:1984}
implies that any element $\xi \in \ss^*$
can be conjugated to
 $-i \aa_+$ by the
$\Ad^*_K$-action,
and therefore also
by the $\Ad^*_U$-action, 
to a unique element so that we have $$
\mu(T^*(U/K)) \cap \tt^*_+ = \tt_+^* \cap \ss^* = -i \aa^*_+    \, . 
$$
%But from \cite[Theorem V.4.1]{helgason:1984}
%we conclude that 
%$$
%{\rm Cone}(\Lambda^K_+) = \aa_+^*,
%$$
%and also that the relative interior of $\aa^*_+$ is orthogonal %to the simple 
%roots of $\Phi_0$ and only to those, thus proving that 
%${\rm codim}(\tau_{U/K})=r_0$. 
%Equality (\ref{e-role-of-arrows})
%follows from the classification of symmetric spaces given by %the Satake diagrams (see 
%(\ref{e-simple-restricted-roots})
%and the discussion after). 
\end{proof}

\begin{example}
\label{ex-no-arrows}
%From Proposition \ref{e-mom-set} we see that 
%$\tilde T_\mathrm{inv} = T_\mathrm{inv}$ if and 
%only if there are no arrows in the Satake diagram for the symmetric space. On the other %hand, a simple example with a noneffective action of $\tilde T_\mathrm{inv}$ is given by
%the group $K$ considered as a symmetric space,
Let us consider the example of the group $K$ considered as a symmetric space 
$$
K \cong (K \times K)/K_{\rm diag}  \, ,
$$
where the quotient is taken with respect to the  diagonal action
and 
$ K_{\rm diag}  := {\rm diag}(K) =
\left\{(x, x)  \, : \, x \in K   \right\}$.
Then, if $\tt_\kk$ is a Cartan subalgebra of $\kk$,
we have that $\tt := \tt_\kk \oplus \tt_\kk$ is a Cartan subalgebra of 
$U = K \times K$ and we have
\begin{eqnarray*}
\kk_{\rm diag}  &=&
\left\{(X, X) \mid X \in \kk    \right\} \\
\ss &=&
\left\{(X, -X) \mid X \in \kk    \right\} \\
\aa &=& i \left\{(H, -H) \mid H \in \tt_\kk    \right\} \, .
\end{eqnarray*}
The antilinear involution $\sigma$
is defined on $i \tt $
by
\begin{eqnarray*}
 (iX, iX)^\sigma &:=& -i(X,X) \\
 (iY, -iY)^\sigma &:=& i(Y,-Y) \, , \quad X, Y \in \tt_\kk \, ,
\end{eqnarray*}
and therefore, on
$i\tt^*$ takes the form
\begin{equation*}
 (\eta_1, \eta_2)^\sigma = -(\eta_2,\eta_1)   \, , \quad \eta_1, \eta_2 \in i\tt_\kk^*  \,  .
\end{equation*}
Let $r_0$ be the rank of $K$ (so that $r=2r_0$) and consider an order on 
$i\tt_\kk^*$ corresponding to the simple 
roots $\Delta_\kk = \left\{\alpha_1, \dots, \alpha_{r_0} \right\}$ 
for the root system $\Phi_\kk \subset 
i \tt_\kk^*$
and an order on
$i\tt^*$ corresponding to the
following simple roots for $\tt = \tt_\kk \oplus \tt_\kk$,
$$\Delta := \left\{(\alpha_1,0) , \dots, (\alpha_{r_0},0), (0,-\alpha_1), \dots ,  (0, -\alpha_{r_0})\right\}  \, .
$$
Then, 
$$
\Phi = \left\{(\alpha,0), (0,\beta) \mid \alpha, \beta \in \Phi_{\kk}
\right\} \, , 
$$
and
\begin{eqnarray*}
\Phi_+ &=& \left\{(\alpha,0), (0,-\beta) \mid \alpha, \beta \in \Phi_{\kk, +}
\right\} \, , 
\\
\Phi_0 &=& \left\{(\alpha,\beta) \in \Phi \mid (\alpha,\beta)_{|_\aa} = 0 \right\}
=
\emptyset \, .
\end{eqnarray*}
Then,
$$
\Delta_0 = \Delta \cap \Phi_0 =
\emptyset   \, , 
$$
so that there are no black circles
in the Satake diagram and therefore the principal stratum 
is the interior of the Weyl chamber, with $$\dim(\tau_{U/K})=l=r_0   \, .$$
Let us now verify that this order is a $\sigma$--order. 
We have 
that 
\begin{equation}
\label{e-sigma-on-roots}
 (\alpha_j,0)^\sigma = (0,-\alpha_j) , \, \, 
 (0, -\alpha_j)^\sigma = (\alpha_j, 0) \, ,
    \end{equation}
and 
$$
(\Phi \setminus \Phi_0) \cap \Phi_+ = \Phi_+ \, .
$$
Then, since $(\alpha, -\beta)^\sigma = (\beta, -\alpha)$
we find that
$$
\sigma \, : \, 
(\Phi \setminus \Phi_0) 
\cap \Phi_+ = \Phi_+ \longrightarrow \Phi_+ \, ,
$$
so that indeed the condition (\ref{e-sigma-order})
is verified.
From 
(\ref{e-perm-sigma-tilda}) and
(\ref{e-sigma-on-roots})
we see that $\tilde \sigma$ 
simply permutes the simple
roots of one factor with the 
simple roots of the second factor
so that the Satake diagram reads: 
\begin{center}
\begin{tikzpicture}
  \draw[line width=1.5pt] (2,1.5) edge (0,1.5) edge (3,2.5) edge (3,0.5);
  \draw[line width=1.5pt] (2,-1.5) edge (0,-1.5) edge (3,-2.5) edge (3,-0.5);

  \draw[line width=1.5pt, fill, color=white] (0,1.5) circle [radius=3pt];
  \draw[line width=1.5pt, fill, color=white] (1,1.5) circle [radius=3pt];
  \draw[line width=1.5pt, fill, color=white] (2,1.5) circle [radius=3pt];
  \draw[line width=1.5pt, fill, color=white] (3,2.5) circle [radius=3pt];
  \draw[line width=1.5pt, fill, color=white] (3,0.5) circle [radius=3pt];
  \draw[line width=1.5pt, fill, color=white] (0,-1.5) circle [radius=3pt];
  \draw[line width=1.5pt, fill, color=white] (1,-1.5) circle [radius=3pt];
  \draw[line width=1.5pt, fill, color=white] (2,-1.5) circle [radius=3pt];
  \draw[line width=1.5pt, fill, color=white] (3,-2.5) circle [radius=3pt];
  \draw[line width=1.5pt, fill, color=white] (3,-0.5) circle [radius=3pt];

  \draw[line width=1.5pt] (0,1.5) circle [radius=3pt];
  \draw[line width=1.5pt] (1,1.5) circle [radius=3pt];
  \draw[line width=1.5pt] (2,1.5) circle [radius=3pt];
  \draw[line width=1.5pt] (3,2.5) circle [radius=3pt];
  \draw[line width=1.5pt] (3,0.5) circle [radius=3pt];
  \draw[line width=1.5pt] (0,-1.5) circle [radius=3pt];
  \draw[line width=1.5pt] (1,-1.5) circle [radius=3pt];
  \draw[line width=1.5pt] (2,-1.5) circle [radius=3pt];
  \draw[line width=1.5pt] (3,-2.5) circle [radius=3pt];
  \draw[line width=1.5pt] (3,-0.5) circle [radius=3pt];

  \draw (-0.1,1.35)  edge[{Stealth[width=6pt,length=10pt]}-{Stealth[width=6pt,length=10pt]}, bend  right] (-0.1,-1.35);
  \draw (0.9,1.35)  edge[{Stealth[width=6pt,length=10pt]}-{Stealth[width=6pt,length=10pt]}, bend  right] (0.9,-1.35);
  \draw (1.9,1.35)  edge[{Stealth[width=6pt,length=10pt]}-{Stealth[width=6pt,length=10pt]}, bend  right] (1.9,-1.35);
  \draw (2.95,2.35)  edge[{Stealth[width=6pt,length=10pt]}-{Stealth[width=6pt,length=10pt]}, bend  right] (2.95,-0.35);
  \draw (3.05,0.35)  edge[{Stealth[width=6pt,length=10pt]}-{Stealth[width=6pt,length=10pt]}, bend  left] (3.05,-2.35);
\end{tikzpicture}
\end{center}

\begin{center}
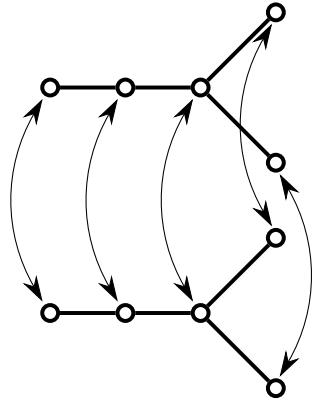
{\it Satake diagram of $\left(\mathfrak{so}(10)\oplus \mathfrak{so}(10), {\mathfrak{so}_{\rm diag}(10)}\right)$}
\end{center}
%Then, from (\ref{e-role-of-arrows}), it follows that
%$\dim \aa = \dim(\tau_{U/K})=r$
%and
For the Guillemin--Sternberg coordinates 
we  therefore have the  relations,
${(\alpha_j, 0)} = - {(0, \alpha_j)}   \, .$
\end{example}

Let us now describe the action of the invariant torus in more detail. Recall from \cite{guillemin.sternberg:1984, lane:2017, knop:2011} that, under mild conditions, if 
$X$ is a smooth manifold with an Hamiltonian $U$-action with equivariant moment map
$\mu:X \to \uu^*$, then along the (open dense) regular stratum $X_\mathrm{reg}= \mu_\mathrm{inv}^{-1}(\tau_X)$, where $\mu_\mathrm{inv}=s\circ\mu$ as in \eqref{eq:mu_inv} and $\tau_X$ is the principal stratum of the Kirwan polytope associated with the $U$-action on $X$, is the moment map for a Hamiltonian action of a torus $T_\mathrm{inv}$,
$$
t \star p = (u^{-1} t u) \cdot p,\,\, p\in X_\mathrm{reg}, \, t\in T_\mathrm{inv},
$$
where $\mathrm{Ad}^*_u \mu(p)\in \tau_X$. In our case, $T^*(U/K)_\mathrm{reg}$, writing $\tau_{U/K}$ as before instead of $\tau_{T^*(U/K)}$, we have that 
$$
\tau_{U/K} = -i\check\aa_+^*,
$$
where 
$$\check\aa^*_+=\{\eta\in \aa^*\mid \forall \beta\in \Sigma_+: \langle \beta,\eta\rangle>0\},$$
and $\Sigma_+:=\{\alpha|_\aa\mid \alpha\in \Phi_+\setminus \Phi_0\}$.
Using the $K$-action, as in the proof of Proposition \ref{e-mom-set}, an element $[x,\xi]\in T^*(U/K)_\mathrm{reg}$ can be written uniquely in the form $[u,\xi_+]$ where $u\in U$ and $\xi_+\in -i\check \aa^*_+.$ The invariant torus is then
$$
T_\mathrm{inv} \cong T_\ss := \exp (i\aa),
$$
with action
$$
[u,\xi_+]\star t = [ut^{-1}, \xi_+], t\in T_\mathrm{inv}.
$$
{}

\subsection{Fourier harmonics for $T_\mathrm{inv}$}

Let $f:T^*(U/K)_\mathrm{reg}\to \CC$. We obtain a Fourier decomposition in terms of characters of $T_\mathrm{inv}$,
$$
f = \sum_{\nu\in \hat T_\mathrm{inv}} \hat f_\nu,
$$
where 
$$
\hat f_\nu (p) = \int_{T_\mathrm{inv}} \chi_\nu (t) f(p\star  t)dt,
$$
and $\chi_\nu$ is the character associated to the weight $\nu$ of $T_\mathrm{inv}$. We have
$$
\hat f_\nu (p\star t) = \chi_\nu(t^{-1}) \hat f_\nu (p).
$$

Let now $g$ be an invariant uniformly convex function on $\uu^*$ and denote by $\hat I_g$ the corresponding $U$-invariant K\"ahler 
structure on $T^*(U/K)$ as described in Theorem \ref{th-4}. The coordinate ring of  $(T^*(U/K),\hat I_g)\cong U_\CC/K_\CC$
is generated by the $\hat I_g$-holomorphic functions
\begin{equation}\label{fs0}
  f^g_{\lambda, v_\lambda^K\otimes v^*}([x,\xi]) := 
\tr (\pi_\lambda (x e^{id_\xi g})v_\lambda^K\otimes v^* ),  
\end{equation}
where $\lambda \in \hat U_K, v^*\in V_\lambda^*.$ The restriction of $f^g_{\lambda, v_\lambda^K\otimes v^*}$ to the regular set $T^*(U/K)_\mathrm{reg}$ can then be conveniently written as
\begin{equation}\label{fs}
    f^g_{\lambda, v_\lambda^K\otimes v^*}([x,\xi_+]) = 
\tr (\pi_\lambda (x e^{id_{\xi_+} g})v_\lambda^K\otimes v^* ),
\end{equation}
with $\xi_+\in -i\check \aa_+^*.$ (Note that, from \cite{baier.hilgert.kaya.mourao.nunes:2023}, this implies that 
$d_{\xi_+}g\in \-i\aa_+.$) 

Let $P_\nu:V_\lambda \to V_\lambda$ denote the projection onto the weight space for the weight $\nu\in \ss^*_\mathbb Z:=\{\lambda\in  \tt^*_\mathbb Z : \lambda|_{i(\mathfrak t\cap\mathfrak k)}=0\}.$ Then, for $A\in \mathrm{End} (V_\lambda)$,
$$
\tr (\pi_\lambda(x e^{id_{\xi_+}g})P_\nu A) = e^{i\langle\nu, d_{\xi_+}g\rangle}\tr (\pi_\lambda(x)P_\nu A).
$$

We then obtain the analog of Proposition 4.5 of \cite{baier.hilgert.kaya.mourao.nunes:2023}.

{}

\begin{proposition}\label{fourierharmonics}
The Fourier harmonics for the $\hat I_g$-holomorphic functions
$f^g_{\lambda, v_\lambda^K\otimes v^*}$ read
$$
\widehat{(f^g_{\lambda, v_\lambda^K\otimes v^*})}_\nu ([x,\xi_+]) = \tr (\pi_\lambda(x e^{id_{\xi_+}g})P_\nu v_\lambda^K\otimes v^*).
$$  
\end{proposition}

\begin{proof}
    This is a special case of the proof of Proposition 4.5 in \cite{baier.hilgert.kaya.mourao.nunes:2023}. Namely, in Proposition 4.5 of \cite{baier.hilgert.kaya.mourao.nunes:2023}, just take $A\in \mathrm{End}(\mathrm{V_\lambda})$ to be
    $$
    A = v_\lambda^K\otimes v^*,
    $$
    and, moreover, consider the case when $\nu \in \ss^*_{\mathbb Z}.$ Note that, in general, the spherical vector $v_\lambda^K$ will have non-zero components along different weight spaces.
\end{proof}

\section{Mabuchi rays and a  mixed polarization on $T^*(U/K)$}
\label{sect-mabuchirays}

\subsection{The polarization ${\mathcal P}_\infty$ on $T^*(U/K)_\mathrm{reg}$}
\label{sec-fibering}

In this section, we will apply the concept of fibering polarization, developed in Section 3 of \cite{baier.hilgert.kaya.mourao.nunes:2023}, to the symplectic manifold $T^*(U/K)_\mathrm{reg}$. This will provide a generalization of the Kirwin--Wu polarization on $(T^*K)_\mathrm{reg}$, $\mathcal{P}_\mathrm{KW}$, which was studied in detail in \cite{baier.hilgert.kaya.mourao.nunes:2023}, and which corresponds to the case $U=K\times K$ and $U/K \cong K.$ 

Consider the diagram
\begin{equation*}
\begin{tikzcd} 
		& \begin{array}{l}
		U\times_K \ss^\ast_{\mathrm{reg}} \\
		\quad \parallel\\
		T^\ast (U/K)_{\mathrm{reg}}
		\end{array}
		  \ar[rd, "\mu_\mathrm{inv}"] \ar[ld, swap, "\mu"] &\\
		\mathfrak{u}_{-i\aa^\ast}^\ast  \ar[rr,"\phi"] & & -i\aa^{\ast}_+
\end{tikzcd}
\end{equation*}
where $\mathfrak{u}^\ast_{-i\aa^\ast}\subset \uu^*$ is the subset of elements whose coadjoint orbits intersect 
$\tt^\ast$ in $-i\aa^\ast$ and $\phi$ is the restriction of the sweeping map $s$ to this set. As detailed in \cite{baier.hilgert.kaya.mourao.nunes:2023}, this describes a $U$-invariant mixed polarization $\mathcal{P}_\infty$, of $T^*(U/K)_\mathrm{reg}$ whose real directions are given by the orbits of $T_\mathrm{inv}$ and whose complex directions are given by the fibers of the map $\phi$.  In our case, these are coadjoint orbits
$$
\phi^{-1} (\xi_+) = \mathcal{O}_{\xi_+}, \, \xi_+\in -i\aa_+^\ast,
$$
which carry a canonical complex structure given by the identifications
$$
\mathcal{O}_{\xi_+} \cong U_\CC / B,
$$
where $B$ is a Borel subgroup. The fibers of $\mu_\mathrm{inv}$ define the coisotropic distribution 
$$
{\mathcal E} := (\mathcal{P}_\infty + \overline{\mathcal{P}_\infty})\cap T(T^*(U/K)_\mathrm{reg}).
$$
Each of the coadjoint orbits $\mathcal{O}_{\xi_+},\, \xi_+\in -i\aa_+^\ast$ is then the base of a $T_\mathrm{inv}$-principal fiber bundle given by the corresponding fiber of $\mu_\mathrm{inv}$ and corresponds to its coisotropic reduction. 

Since $T^*(U/K)_\mathrm{reg}$ is multiplicity free (the orbits of the $U$-action are separated by the values of $\mu$) we can apply Corollary 3.19 of \cite{baier.hilgert.kaya.mourao.nunes:2023} so that any $\mathcal{P}_\infty$-polarized section will be supported on the inverse image under $\mu_\mathrm{inv}$ of the intersection of the interior of the Kirwan polytope with $-i$ times the  $\rho$-translate of the character lattice of the maximal torus, where $\rho$ is the half-sum of positive roots.

The notation $\mathcal{P}_\infty$ is justified below when we show that 
$\mathcal{P}_\infty$ can be obtained at infinite geodesic time along a Mabuchi ray of $U$-invariant K\"ahler polarizations of $T^*(U/K)$.

\subsection{Local description of ${\mathcal P}_\infty$}
Let us describe the polarization $\mathcal{P}_\infty$ locally in terms of Hamiltonian vector fields. Consider the functions
$$
F_{\lambda, v_\lambda^K\otimes v^*}([x,\xi_+]):=
\tr (\pi_\lambda(x)P_\lambda v_\lambda^K\otimes v^*), \,\, [x,\xi_+]\in T^*(U/K),
$$
for $v^*\in V_\lambda^*$.
\begin{proposition}\label{prop-holomorphicdirections}
    Let $V_\lambda$ be a highest weight representation of $U$ with highest weight $\lambda\in \hat U_K.$ 
    For $v^*_1,v_2^*\in V_\lambda^*$, the complex function
    $$
    \frac{F_{\lambda, v_\lambda^K\otimes v^*_1}}{F_{\lambda, v_\lambda^K\otimes v^*_2}},
    $$
    defined on the subset $O\subset T^*(U/K)_\mathrm{reg}$ where the denominator does not vanish, is $T_\mathrm{inv}$-invariant and hence descends to a function on $O/T_\mathrm{inv}$ on an open subset of  $\ss^*.$ The family of such functions, for all possible choices of $v_1^*, v_2^*\in V_\lambda^*$ generate the complex directions of the polarization $\mathcal{P}_\infty.$ 
\end{proposition}

\begin{proof}
This analogous to the proof of Proposition 4.2 of \cite{baier.hilgert.kaya.mourao.nunes:2023}. Namely, the functions $F_{\lambda, v_\lambda^K\otimes v^*}$ are $T_\mathrm{inv}$-equivariant,
$$
F_{\lambda, v_\lambda^K\otimes v^*}([x,\xi]\star t) = \chi_\lambda(t^{-1}) 
F_{\lambda, v_\lambda^K\otimes v^*}[x,\xi],\,\, t\in T_\mathrm{inv},
$$
so that their quotients will be invariant under $T_\mathrm{inv}.$ Moreover, we see that the functions 
$F_{\lambda, v_\lambda^K\otimes v^*}$ are defined on 
$U_\CC$, are $B$-equivariant  and that therefore they define  holomorphic sections of the ample 
Borel--Weil line bundle $L_\lambda\to \mathcal{O}_{\xi_+}$ on $U_\CC/B$, where $B\subset U_\CC$ is the Borel subgroup associated to the choice of positive roots.
For fixed $\lambda$ and varying $v^*\in V_\lambda^*$ we thus obtain 
$$H^0(\mathcal{O}_{\lambda},L_\lambda)\cong V_\lambda^* \cong \left\{F_{\lambda, v_\lambda^K\otimes v^*}\mid v^*\in V_\lambda^*\right\}.$$ The quotients therefore generate the structure sheaf of $\mathcal{O}_{\xi_+}$ and the Hamiltonian vector fields of their complex conjugates define an holomorphic polarization along the coisotropic reductions.  
\end{proof}

As described in Section \ref{sec-fibering}, the remaining  real directions of the mixed polarization $\mathcal{P}_\infty$ are generated by the Hamiltonian vector fields of the Guillemin--Sternberg action coordinates given by the components of $\mu_\mathrm{inv}.$

\subsection{Convergence of polarizations}
\label{subs-convpol}

We will now describe how the mixed polarization $\mathcal{P}_\infty$ arises at infinite geodesic time along a Mabuchi ray of $U$-invariant K\"ahler polarizations $\mathcal{P}_{g_t}, t>0$, associated to the $U$-invariant K\"ahler structures $\widehat I_{g_t}$, where $g_t:=g+th$ for $g,h\in \Conv^\infty_{\rm unif}(\uu^\ast)^{\Ad^\ast_U}$. 

\begin{lemma}\label{lemma-convpol}
For the Mabuchi ray of $U$-invariant K\"ahler polarizations on $T^*(U/K)_\mathrm{reg}$ given by the complex structures $\hat I_{g_t}, t>0$, one has, for $\lambda\in \hat U_K, v^*\in V_\lambda^*$:
\begin{enumerate}
\item[(i)]
$
\lim_{t\to +\infty} e^{-i\langle\lambda, \mathcal{L}_{g_t \circ \mu_\mathrm{inv}}\rangle }f^{g_t}_{\lambda, v_\lambda^K\otimes v^*}  = 
F_{\lambda, v_\lambda^K\otimes v^*},
$

\item[(ii)]
$
\lim_{t\to +\infty} \frac{1}{t} \mathrm{ln} f^{g_t}_{\lambda, v_\lambda^K\otimes v^*} = 
i\langle\lambda, \mathcal{L}_h\circ \mu_\mathrm{inv}\rangle.
$
\end{enumerate}
\end{lemma}

\begin{proof}
    This follows from Lemma 4.8 and equation $(30)$ in \cite{baier.hilgert.kaya.mourao.nunes:2023}.
\end{proof}

\begin{theorem}
    The family of K\"ahler polarizations $\mathcal{P}_{g_t}, t>0,$ on $T^*(U/K)_\mathrm{reg}$ converges pointwise to the mixed polarization $\mathcal{P}_\infty$ as $t\to +\infty.$
\end{theorem}

\begin{proof}
For the real directions of $\mathcal{P}_\infty$, we see from (ii) in Lemma \ref{lemma-convpol} that the distribution generated by the vector fields
$$
(f^{g_t}_{\lambda, v_\lambda^K\otimes v^*})^{-1} X_{f^{g_t}_{\lambda, v_\lambda^K\otimes v^*}},
$$
for $\lambda\in \hat U_K, v^*\in V_\lambda^*$, 
converges, as $t\to +\infty$, pointwise to the distribution generated by the Hamiltonian vector fields of the components of $\mu_\mathrm{inv}.$ On the other hand, from (i) in Lemma \ref{lemma-convpol}, we see that
the distribution generated by the Hamiltonian vector fields
$$
X_{\frac{f^{g_t}_{\lambda, v_\lambda^K\otimes v^*}}{f^{g_t}_{\lambda, v_\lambda^K\otimes w^*}}},\, 
\lambda\in \hat U_K, v^*,w^*\in V_\lambda^*,
$$
converges, as $t\to +\infty$, to the distribution generated by the Hamiltonian vector fields
$$
X_{\frac{F_{\lambda, v_\lambda^K\otimes v^*}}{F_{\lambda, v_\lambda^K\otimes w^*}}},\, 
\lambda\in \hat U_K, v^*,w^*\in V_\lambda^*
$$
which, from Proposition \ref{prop-holomorphicdirections}, gives the holomorphic directions of $\mathcal{P}_\infty.$
\end{proof}

\subsection{The half-form bundle for ${\mathcal P}_\infty$}
\label{subsec-half}

In this section, we will study the half-form bundle for $\mathcal{P}_\infty.$ To do so, let us 
describe the sections of the canonical bundle for the K\"ahler polarizations $\mathcal{P}_{g_t}, t>0.$ 

Recall from \cite{akhiezer:1986} that homogeneous  holomorphic vector bundles on $U_\CC/K_\CC$, that is vector bundles with a lift of the $U_\CC$-action on the base, are in bijective correspondence with holomorphic representations of $K_\CC.$ Sections of a homogeneous line bundle $L_\chi \to U_\CC/K_\CC$, defined by the character $\chi$ of $K_\CC$, are therefore 
in bijection with $\chi$-equivariant functions on $U_\CC$, that is with functions
$f:U_\CC \to \CC$ such that 
$$
f(u\cdot k) = \chi(k)^{-1} f(u),\, u\in U_\CC, k\in K_\CC.
$$
The canonical bundle $\mathcal{K}$ of $U_\CC/K_\CC$ is associated to the character defined by, see \cite{akhiezer:1986},
$$
\det \left(\Ad_h: \ss_\CC \to \ss_\CC\right), \, h\in K_\CC.
$$
Since this determinant gets a contribution from both positive and negative roots for the adjoint action of $\kk_\CC$ it defines the trivial character. (Note that these roots are with respect to a choice of Cartan subalgebra for $\uu$ which extends a Cartan subalgebra for $\kk$ and which in general will be different from $\tt$.) 
Therefore, $\mathcal{K}$ is trivializable.

Recall from Theorem \ref{th-4} that for $g\in \Conv^\infty_{\rm unif}(\uu^\ast)^{\Ad^\ast_U}$ the Legendre transform $\mathcal{L}_g$ provides a biholomorphism between $(T^*(U/K), \hat I_g)$ and $U_\CC/K_\CC$,
 and therefore also the canonical bundle $K_{g}:=\mathcal{L}_g^* \mathcal{K}$ of $(T^*(U/K), \hat I_g)$, is trivializable.

Holomorphic trivializing sections for $\mathcal{K}$ and $K_g$ can be obtained by making use of the transitive holomorphic left action of $U_\CC$, as follows. (See Proposition 3.3.22 in \cite{kaya:2015}.)
On a sufficiently small open neighborhood $A$ of 
$[e]\in U_\CC/K_\CC$ one can choose a local section $s:A \to U_\CC$ of the canonical projection $p:U_\CC \to U_\CC/K_\CC$, identifying 
$$
T_{[e]} (U_\CC/K_\CC) \cong \ss \oplus i\ss \cong \ss_\CC.
$$
Given a basis of $\ss^*$, one can then define a holomorphic form of top degree on $T^*_{[e]}(U_\CC/K_\CC)$. When $K$ is connected, as in the case that we are considering, this form can then be extended to a global holomorphic form of top degree, $\Omega$, on the whole of $U_\CC/K_\CC$ by imposing invariance under the holomorphic transitive left action of $U_\CC$. (Connectedness of $K$ ensures that the right action of $K_\CC$ preserves $\Omega$, see Proposition 3.2.33 in \cite{kaya:2015}.) $\Omega$ then defines a global holomorphic frame for $\mathcal{K}\to U_\CC/K_\CC$, so that 
$\mathcal{L}_g^* \Omega$
is a global $U_\CC$-left invariant holomorphic trivializing frame for $K_g\to T^*(U/K)$.

Letting $\dim K =n, \dim U = d+n$, let us consider a basis of $U$-left-invariant one-forms on $U$, $\omega^j, j=1, \dots, d+n$, dual to an orthonormal basis of $\uu$ adapted to the decomposition $\uu = \kk \oplus \ss$ that is, such that $(\omega^j)_{\vert_{\kk}}=0$ for $j=1, \dots d$ and $(\omega^j)_{\vert_{\ss}}=0$ for $j=d+1, \dots, d+n$. 

\begin{lemma}\label{leftinvforms} (Lemma 3.3 in \cite{kirwin.mourao.nunes:2013})
    A $U_\CC$-left-invariant holomorphic frame of one-forms on $(T^*U,I_g)$ is given by
    $$
    \Omega^j_g (x,\xi):= \sum_{k=1}^{d+n} \left[
    e^{-i \mathrm{ad}_{d_{\xi}g}}\right]_k^j 
    \omega^k + \left[
    \frac{1-e^{-i \mathrm{ad}_{d_{\xi}g}}}{\mathrm{ad}_{d_{\xi}g}}\cdot \mathrm{Hess}_g(\xi)\right]_k^j d \xi^k,
    $$
    $j=1, \dots, d+n.$
    This is obtained by pulling-back a frame of left $U_\CC$-invariant holomorphic one-forms on $U_\CC$ by the Legendre transform 
    $\mathcal{L}_g.$
\end{lemma}
{}

The left actions of $U_\CC$ on itself and $U_\CC/K_\CC$, and the  corresponding actions on $T^*U$ and $T^*(U/K)$ induced by the respective Legendre transforms $\mathcal{L}_g$ are, of course, all compatible.
We have then, 

\begin{proposition}\label{prop-leftinvcanonical} On $T^*U$,
\begin{equation}\label{leftinvcanonical}
\Omega_g := (\mathcal{L}_g^* \circ p^*) \Omega = \bigwedge_{j=1}^d \Omega^j_g.
\end{equation}
\end{proposition}

{}

\begin{proof}
The forms on both sides of (\ref{leftinvcanonical}) are $U_\CC$-left-invariant and are therefore defined by their values at $(e,0)\in 
T^*U$, where $e\in U$ is the identity. 
From Remark \ref{gissigmacompatible} we have that 
$\xi\in \ss^* \iff d_\xi g \in \ss$ and $\xi\in \kk^* \iff d_\xi g\in \kk^*$. This implies, in particular, that $d_{\xi=0}g=0$.\footnote{Recall that we are assuming that $g$ is compatible with the symmetric space involution $\sigma$ as in Remark \ref{gissigmacompatible}.} From (c) in Proposition 2.7 in \cite{baier.hilgert.kaya.mourao.nunes:2023}, we have that, as endomorphisms of $\uu^*$, 
$$
\ad^*_{\xi^*} = \Hess^{-1}_g (\xi) \ad^*_{d_\xi g},
$$
where $\xi^*\in \uu$ corresponds to $\xi\in \uu^*$ via the invariant form on $\uu$.
It follows that the endomorphism $\Hess_g(\xi)$ preserves the decomposition $\uu^*=\kk^* \oplus \ss^*$ for any $\xi\in \uu^*$. Therefore, since $[\ss,\kk]\subset \kk$ and $[\ss,\ss]\subset \kk$, we have that at $(e,0)$, for $j=1, \dots d,$
$$
\Omega^j_g (e,0)= \omega^j + i \sum_{k=1}^d [\Hess_g(0)]^j_k d\xi^k = 
(\mathcal{L}_g^* \circ p^*)\Omega^j (e), 
$$
where $e$ also denotes the identity in $U_\CC$ and  
where $\Omega^j$ is the $U_\CC$-left-invariant 1-form on $U_\CC/K_\CC$ such that $(p^* \Omega^j)(e) = d\xi^j + id\xi^j$. Then, from Proposition 3.3.22 in \cite{kaya:2015}, we get 
$$
\Omega = \bigwedge_{j=1}^d \Omega^j.
$$
This ends the proof.
\end{proof}

\begin{remark}
    Note that each individual $\Omega^j_g$ is not the pull-back of an holomorphic one-form on $U_\CC/K_\CC$ since $K_\CC$ will not act trivially on it. 
\end{remark} 

{} 

Note that, while the form $\Omega$ is fixed, the left $U_\CC$ actions on $T^*U$ and  $T^*(U/K)$ are induced by the Legendre transforms $\mathcal{L}_g$ and therefore depend on the choice of $g$. For that reason, to study the behavior of $\Omega_{g+th}$ as $t\to +\infty$ for the families of K\"ahler structures considered in Section \ref{subs-convpol}, it is not enough to consider the restriction to the identity coset $[e]\in U_\CC/K_\CC.$ On the other hand, the restriction of $p\circ \mathcal{L}_g$ to $U\times \ss^*$ is surjective onto $U_\CC/K_\CC$ so that it is enough to study the form 
$\Omega_g$ in (\ref{leftinvcanonical}) along $U\times \ss^*.$ Along $U\times \ss^*$, since $j$ also runs from $1$ to $d$, only terms with even powers of $\mathrm{ad}_{d_{\xi}g}$ contribute, again due to the conditions $[\ss,\kk]\subset \kk$ and $[\ss,\ss]\subset \kk$. We can, moreover, using the $\mathrm{Ad}^*_K$-action, restrict the study of the behavior of $\Omega_g$ even further to $U\times (-i)\aa^*_+.$

Let, for $j=1, \dots, d$ and $\xi_+\in (-i)\aa^*$,
\begin{equation}\label{leftformsdown}
    \tilde\Omega^j_g (x,\xi_+):= \sum_{k=1}^d \left[\cosh
     (i\mathrm{ad}_{d_{\xi}g})\right]_k^j 
    \omega^k + \left[
    \frac{\sinh(i \mathrm{ad}_{d_{\xi}g})}{\mathrm{ad}_{d_{\xi}g}}\cdot \mathrm{Hess}_g(\xi)\right]_k^j d\xi^k, 
\end{equation}
so that over $U\times (-i)\aa^*$ we have
\begin{equation}\label{tildepullback}
\Omega_g = \bigwedge_{j=1}^d \tilde \Omega^j_g. 
\end{equation}

\begin{remark}Note that under the involution 
$(\mathrm{Id_U\times \sigma)}$ on $T^*U$,  which descends to the anti-holomorphic involution $[x,\xi]\mapsto [x,-\xi]$ on $T^*(U/K)$, see \cite{stenzel:1990}, and using Remarks \ref{gissigmacompatible}, \ref{legendreok} and \ref{rmk-legendreok2}, we do obtain that $(\mathrm{Id}_U\times \sigma)^* \tilde \Omega^j_g =\overline{\tilde \Omega^j_g}.$ On the other hand, the involution $\sigma$ on $\uu$ lifts to a holomorphic involution $\sigma$ of $T^*(U/K)$ which acts as $[x,\xi]\mapsto [\sigma(x),-\xi]$ where $x\mapsto \sigma(x)$ is the involution on $U$ induced from $\sigma$ by the exponential map. In particular, at the points $(e,\xi)$ it is easy to check that the pull-back of $\tilde\Omega^j_g$ under the lift of this involution to $T^*U$ is $-\tilde\Omega^j_g$, consistently with $\sigma$ being holomorphic. 
\end{remark}
{}
We then have, with $g_t = g+th, t\geq 0$, 

\begin{proposition}\label{limitcanonicalsection}
On $T^*U_\mathrm{reg} \cap (U\times (-i)\aa^*)$, 
\begin{equation}\label{limitcanform}
\lim_{t\to +\infty} t^{-l} e^{-2 \langle\hat\rho, d_{\tilde \xi_+}g_t\rangle} \Omega_{g_t} (x,\xi_+)= 
c_0 i^l 2^{-\#(\Sigma\cap \Phi_+)} \det \left(\mathrm{Hess}_h(\tilde \xi_+)\right)_{i\aa}  \hat\Omega_\infty,
\end{equation}
where $\hat\rho$ is the half-sum of the positive restricted roots weighted by multiplicity, $c_0$ is a nonzero constant, and 
where
\begin{equation}\label{tildeomegainfty}
\hat \Omega_\infty:= \bigwedge_{j=1}^l d \xi^j\cdot
\bigwedge_{\alpha\in \Sigma\cap \Phi_+}(\omega^\alpha-\omega^{\alpha^\sigma}-\langle\alpha,\xi_+\rangle^{-1} (d\xi^\alpha-d\xi^{\alpha^\sigma})).
\end{equation}
The right-hand sides of (\ref{limitcanform}) and (\ref{tildeomegainfty}) 
define smooth form on $T^*U_{\mathrm{reg}}$ by the $K$ right-invariant extension from $U\times (-i)\aa^*.$
\end{proposition}

\begin{proof}
If $\eta\in i\aa$, we have
$$
\mathrm{ad}_\eta^2 (E_\alpha-E_{\alpha^\sigma}) =\langle \beta, \eta\rangle^2 (E_\alpha - E_{\alpha^\sigma}),
$$
where $\beta$ is the restricted root obtained by restricting $\alpha$ to $\aa.$ (Recall from Section \ref{subsection_restrictedroots} that for restricted roots $\alpha^\sigma_{\vert_{\mathfrak a}}= \alpha_{\vert_{\mathfrak a}}$.) We then have, from (\ref{leftformsdown}), the proof of Lemma 4.14 in \cite{baier.hilgert.kaya.mourao.nunes:2023} and from Remarks \ref{gissigmacompatible} and \ref{legendreok}, for regular $\xi_+$,
$$
\lim_{t\to +\infty}t^{-l}\bigwedge_{j=1}^l \tilde \Omega^j_{g_t} (x,\xi_+)= i^l \det \left(\mathrm{Hess}_h(\xi_+) \right)_{i\aa}d\xi^1\wedge \cdots\wedge d\xi^l 
$$ 
and
$$
\lim_{t\to +\infty} e^{-2 \langle\hat\rho, d_{\xi_+}g_t\rangle} 
\bigwedge_{\alpha\in \Sigma\cap \Phi_+}\left(\tilde\Omega^\alpha_{g_t} -\tilde\Omega^{\alpha^\sigma}_{g_t}\right)(x,\xi_+) = 
$$
$$
=2^{-\#(\Sigma\cap \Phi_+)} 
\bigwedge_{\alpha\in \Sigma\cap \Phi_+} 
(\omega^\alpha-\omega^{\alpha^\sigma}-\langle\alpha,\xi_+\rangle^{-1} (d\xi^\alpha-d\xi^{\alpha^\sigma}).
$$
The constant $c_0$ comes from the fact that we are changing from an orthonormal basis in (\ref{leftformsdown}) to a Chevalley basis. 
\end{proof}

Let $\tilde p: U\times \ss^* \to U\times_K \ss^*$ be the canonical projection. 

\begin{proposition}
    The form $\hat \Omega_\infty$ is the pull-back by $\tilde p$ of a 
    trivializing section $\tilde\Omega_\infty$ of the canonical bundle of $\mathcal{P}_\infty$ over $T^*(U/K)_\mathrm{reg}$, $$\hat \Omega_\infty= \tilde p^* \tilde \Omega_\infty.$$
\end{proposition}
    
\begin{proof}
    We have $p\circ \mathcal{L}_{g_t} = \mathcal{L}_{g_t} \circ \tilde p,\, t>0,$ so that $\mathcal{L}_{g_t}^* \circ p^* = \tilde p^* \circ \mathcal{L}_{g_t}^*.$ From Proposition \ref{prop-leftinvcanonical} and since (\ref{limitcanform}) extends by right $K$-invariance we have that 
    the left-hand side of (\ref{limitcanform}) can be written as the limit as $t\to \infty$ of a pull-back by $\tilde p \circ \mathcal{L}_{g_t}$ from $U_\CC/K_\CC$. But since $\tilde p$ is a submersion this implies that the right-hand side of (\ref{limitcanform}) is the pull-back under $\tilde p$ of a well defined form on $T^*(U/K)_\mathrm{reg}$ whence
    $\hat \Omega_\infty = \tilde p^*\tilde\Omega_\infty$.

    Moreover, since  $\lim_{t\to \infty} \mathcal{P}_{g_t}=\mathcal{P}_\infty,$ pointwise in the Lagrangian Grassmannian over $T^*(U/K)_\mathrm{reg}$ we obtain that $\tilde \Omega_\infty$, which is never vanishing, gives a trivializing section of the corresponding canonical bundle.
\end{proof}

We also have,

\begin{remark}
    Note that, from Theorem 4.1 in Chapter 5 in \cite{helgason:1984}, as recalled in  Section \ref{subsection_sphericalreps}, we obtain that $\hat\rho$ is a spherical weight. Moreover, one can check that the representation $V_{\hat\rho}$ is self-dual so that under the 
    identification $V_{\hat\rho}\cong V_{\hat\rho}^*$ via an invariant inner product, since the subspace of spherical vectors is one-dimensional,  we can take $(v_{\hat\rho}^K)^*=v_{\hat\rho}^K$.
\end{remark}

\begin{proposition}\label{polhalfform}
 On the subset of $\check W$ where $F_{\hat \rho, v_{\hat\rho}^K\otimes v_{\hat\rho}^K}$ is nonzero, 
 $$
 \Omega_\infty:= F_{\hat \rho, v_{\hat\rho}^K\otimes v_{\hat\rho}^K}^{-2}\tilde\Omega_\infty
 $$
 is a polarized section of $\mathcal{P}_\infty.$
\end{proposition}

\begin{proof}
This follows from Proposition \ref{limitcanonicalsection}
and from Lemma \ref{lemma-convpol}, since both $\Omega_{g_t}$ and $f^{g_t}_{\hat \rho, v_{\hat\rho}^K\otimes v_{\hat\rho}^K}$ are $\mathcal{P}_{g_t}$-polarized for all $t>0$.
\end{proof}

Since the canonical bundle of $\mathcal{P}_\infty$ over $T^*(U/K)_\mathrm{reg}$ is trivializable, with trivialing section $\tilde \Omega_\infty$, we can define a bundle of half-forms by taking its square root with trivializing section $\tilde \Omega_\infty^\frac12$.

{}

{}

{}
\section{Quantizations of $T^*(U/K)$ along Mabuchi rays}
\label{sect-extendedquantumbundle}

In this section we will describe the quantizations of $T^*(U/K)$ for the vertical (or Schr\"odinger) polarization and for the $U$-invariant K\"ahler polarizations determined by the complex structures 
$\hat I_g, g\in \Conv^\infty_{\rm unif}(\uu^\ast)^{\Ad^\ast_U}.$ We will then relate them to the space of polarized sections for the limit polarization $\mathcal{P}_\infty$ which, as we will describe, inherits a natural Hilbert space structure, $\mathcal{H}_\infty,$ 
so that we obtain an extended bundle of quantum Hilbert spaces 
$$
\mathcal{H} \to \left\{0\right\}\cup \Conv^\infty_{\rm unif}(\uu^\ast)^{\Ad^\ast_U} \cup \left\{\infty \right\},
$$
equipped with a natural $U$-invariant flat connection.

\subsection{The Schr\"odinger and the K\"ahler quantizations for $\mathcal{P}_{g_t}, t>0$}
\label{subsec_quantizations}

The quantization of $T^*(U/K)$ in the  Schr\"odinger, or vertical, polarization $\mathcal{P}_{\mathrm{Sch}}$ is given by 
\begin{equation}
\mathcal{H}_{\mathrm{Sch}} := L^2(U/K)\otimes \sqrt{dx}
\end{equation}
where $dx$ denotes the pull-back to $T^*(U/K)$, by the canonical projection, of a unit-volume $U$-invariant volume form on $U/K$. We have then, 
$$
\mathcal{H}_{\mathrm{Sch}} =  \widehat \bigoplus_{\lambda\in \hat U_K} {\left\{  \sigma^0_{\lambda,v^*}\mid \lambda\in \hat U_K, v^*\in V_\lambda^*\right\}},
$$
where
$$
\sigma^0_{\lambda,v^*} (x):= \tr (\pi_\lambda(x) v_\lambda^K\otimes v^*)\otimes \sqrt{dx},
$$
and where the hat denotes norm completion with respect to the (usual) $U$-invariant $L^2$ inner 
product, 
$$
\langle \sigma_{\lambda,v^*}^0, \sigma_{\tilde \lambda, \tilde v^*}^0\rangle =
d_\lambda^{-1}\delta_{\lambda\tilde\lambda} \langle v^*, \tilde v^*\rangle_{V_\lambda^*}.
$$

{}

From Theorem~\ref{th-4}, the K\"ahler structure of $T^*(U/K)$ for the complex structure $\hat I_g$, for $g\in \Conv^\infty_{\rm unif}(\uu^\ast)^{\Ad^\ast_U}$, is obtained from a K\"ahler quotient of $(T^*U,\hat I_g)$.  From 
equations (\ref{fs0}), (\ref{fs}) and Section \ref{subsec-half} we also obtain that the half-form corrected quantization of $T^*(U/K)$ in the holomorphic quantization associated to the complex structure $\hat I_g$,  is given by the Hilbert space
\begin{equation}
    \mathcal{H}_{g} := \widehat \bigoplus_{\lambda\in \hat U_K}{\left\{  \sigma^g_{\lambda,v^*}\mid\lambda\in \hat U_K, v^*\in V_\lambda^*\right\}},
\end{equation}
where
\begin{equation}\label{sinatra}
    \sigma^g_{\lambda,v^*} := e^{-g(\lambda+\hat\rho)} f^g_{\lambda, v_\lambda^K\otimes v^*}e^{-\frac12\kappa_g}\otimes \Omega_g^{\frac12},
\end{equation}
and where the inner product in $\mathcal{H}_g$ is the usual inner product for half-form corrected K\"ahler quantization (see, for instance, Section 4 of \cite{kirwin.mourao.nunes:2013}).
Of course, this is an instance of the celebrated results of Guillemin and Sternberg \cite{guillemin.sternberg:1982}. (While $T^*(U/K)\cong U_\CC / K_\CC$ is not compact the quotient can be described in terms of an action of $K_\CC$.) 

As will be discussed below in Section \ref{subsect-extendedquantumbundle}, the first exponential factor in (\ref{sinatra}) is justified by the fact that 
\begin{equation}\label{frame}
\left\{\sigma^g_{\lambda,v_j^*}\mid \lambda\in \hat U_K, v_j^*\in V_\lambda^*, j=1,\dots ,\dim V_\lambda \right\}, 
\end{equation}
with $\left\{v_j\right\}_{j=1, \dots, \dim V_\lambda}$ a basis of $V_\lambda$,
gives a parallel frame for a naturally defined flat connection on the extended quantum bundle.

{}

\subsection{The flat connection on the quantum bundle}
\label{subsect-extendedquantumbundle}

In the previous subsection, for the family of $U$-invariant K\"ahler quantizations of $T^*(U/K)$, we obtained a quantum bundle of Hilbert spaces of polarized sections, 
with fiber $\mathcal{H}_g$ over $g\in\Conv^\infty_{\rm unif}(\uu^\ast)^{\Ad^\ast_U}.$ In the spirit of \cite{axelrod-dellapietra-witten91} (see also \cite{florentino.matias.mourao.nunes:2005, kirwin.mourao.nunes:2014}), this bundle comes equipped with a flat connection, $\nabla^Q$,
which relates the different $U$-invariant K\"ahler quantizations. This connection is obtained by covariantly differentiating sections of the quantum bundle along a tangent vector $h$ to $\Conv^\infty_{\rm unif}(\uu^\ast)^{\Ad^\ast_U}$ by combining the prequantum operator $\hat h$ and a fiber-preserving ``quantum" operator $\mathcal{Q}(h)$. In the case of the quantization of symplectic vector spaces along translation invariant K\"ahler polarizations, this combination defines a flat connection whose (unitary) parallel transport intertwines natural representations of the Heisenberg group on the different fibers.

For $h\in \Conv^\infty_{\rm unif}(\uu^\ast)^{\Ad^\ast_U}$, consider the half-form corrected Kostant-Souriau prequantum operator (let $\nabla$ denote the connection on the prequantum bundle obtained from the prequantum data on $T^*U$ by K\"ahler reduction, \`a la Guillemin-Sternberg)
$$
\hat h := (i\nabla_{X_h} +h) \otimes 1 + 1\otimes iL_{X_h}.
$$
Consider also the quantum operator $\mathcal{Q}(h)$ defined by 
$$
\mathcal{Q}(h) \sigma_{\lambda,v^*}^g := h(\lambda+\hat \rho) \sigma_{\lambda,v^*}^g,
$$
for any $g\in \Conv^\infty_{\rm unif}(\uu^\ast)^{\Ad^\ast_U},\, \lambda \in \hat U_K, v^*\in V_\lambda^*.$ As in Section 2.4 of \cite{florentino.matias.mourao.nunes:2005} and Section 5.3 of \cite{baier.hilgert.kaya.mourao.nunes:2023}, consider sections of the quantum bundle of the form
$$
s(g,[x,\xi]) = f(g, xe^{id_\xi g}) e^{-\frac12\kappa_g}\otimes \Omega_g^\frac12,
$$
such that, for fixed $g$, $f$ is holomorphic in $U_\CC/K_\CC$.  The connection $\nabla^Q$ is then defined by
$$
\nabla_h^Q s := \frac{\delta}{\delta h} s + \hat h s - \mathcal{Q}(h) s
$$
where  $h$ is a tangent vector to $\Conv^\infty_{\rm unif}(\uu^\ast)^{\Ad^\ast_U}$. (This is an analog of equation (1.30) in \cite{axelrod-dellapietra-witten91}.) The evaluation of the connection form, in this frame, along the tangent vector $h$ is therefore given by $(\hat h - \mathcal{Q}(h))$. Note that the quantum operator $\mathcal{Q}$ acts through the isotypical decomposition of $s$ under $U$-representations, corresponding to the decomposition of $s$ in the frame. 

For $T^*U$, as described in Section 5.3 of \cite{baier.hilgert.kaya.mourao.nunes:2023}, the equation of covariant constancy determined by $\nabla^Q$ (along quadratic $h$) corresponds to the heat equation described in \cite{florentino.matias.mourao.nunes:2005}. (The first order differential operator in the connection form, given by the prequantum operator $\hat h$, does not show up explicitly in the heat equation in  \cite{florentino.matias.mourao.nunes:2005} because there the equation was written in $U_\CC$, not in $T^*U$, so the that term corresponding to the Legendre transform $\mathcal{L}_{g+th}$ is implicit there.) Also for quadratic $h$, the condition of covariant constancy is an analog of the heat equation satisfied by theta functions for varying moduli in the quantization of an abelian variety.

{}

Note that the sections $\sigma^{g+th}_{\lambda, v^*}$, for $\lambda\in \hat U_K, v^*\in V_\lambda^*$, which are linear combinations with constant coefficients of the frame sections in (\ref{frame}), 
satisfy the equation of parallel transport:
$$
\frac{d}{dt} \sigma^{g+th}_{\lambda, v^*} = 
\nabla_h^Q \sigma^{g+th}_{\lambda, v^*} = \hat h \sigma^{g+th}_{\lambda, v^*} - h(\lambda+\hat\rho) 
\sigma^{g+th}_{\lambda, v^*},
$$
which follows explicitly from (\ref{sinatra}) and 
from (see Proposition 3.19 in \cite{kirwin.mourao.nunes:2013})
$$
{}
e^{t\hat h}  \left(f^g_{\lambda, v_\lambda^K\otimes v^*}e^{-\frac12\kappa_g}\otimes \Omega_g^{\frac12}\right)
= f^{g+th}_{\lambda, v_\lambda^K\otimes v^*}e^{-\frac12\kappa_{g+th}}\otimes \Omega_{g+th}^{\frac12}.
$$
In particular, this shows that $\nabla^Q$ is flat.

Since the operators $\hat h$ and $\mathcal{Q}(h)$ commute, the parallel transport of $\nabla^Q$ along the Mabuchi ray generated by $h$ is given by exponentiating the connection form, which gives a generalized coherent state transform (gCST) defined by 
$$
C_{t,h} := e^{t\hat h} \circ e^{-t\mathcal{Q}(h)}.
$$
In fact, one can take this parallel transport operator to act also on the Hilbert space of the vertical polarization $\mathcal{H}_\mathrm{Sch}$. The following is an immediate corollary of \cite{kirwin.mourao.nunes:2013} or Section 5.2 of \cite{baier.hilgert.kaya.mourao.nunes:2023}.

\begin{proposition}
    Let $g=0$ or $g\in \Conv^\infty_{\rm unif}(\uu^\ast)^{\Ad^\ast_U}$. For $t\geq 0$, 
    $$
    C_{t,h} \sigma^g_{\lambda,v^*} = \sigma^{g+th}_{\lambda,v^*}.
    $$
     Therefore, $C_{t,h}$ is a $U$-equivariant isomorphism $$C_{t,h}:\mathcal{H}_g\to \mathcal{H}_{g+th},\, t\geq 0.$$ (Here, $\mathcal{H}_0:= \mathcal{H}_{\mathrm{Sch}}.$)
    \end{proposition}

\begin{corollary}
    For $g=0$ or $g\in \Conv^\infty_{\rm unif}(\uu^\ast)^{\Ad^\ast_U}$, letting
\begin{equation}
    \mathcal{H}_{g}^\lambda := {\left\{  \sigma^g_{\lambda,v^*}\mid v^*\in V_\lambda^*\right\}}, \,\, \lambda\in \hat U_K,
\end{equation}
we obtain from $U$-equivariance that there exist constants $a_{t,h,\lambda}\in \RR\setminus \left\{0\right\}$, for $t>0, \lambda\in \hat U_K$, such that
$$
a_{t,h,\lambda}C_{t,h} :  \mathcal{H}_{g}^\lambda \to  \mathcal{H}_{g+th}^\lambda
$$
are unitary isomorphisms.
\end{corollary}
{}

As we describe next, parallel transport along Mabuchi rays of the frame ${\left\{  \sigma^g_{\lambda,v^*}\mid v^*\in V_\lambda^*,\, \lambda\in \hat U_K\right\}}$ has a nice behavior at infinite geodesic time.

\subsection{Convergence of quantizations as $t\to +\infty$}

In this section, in the spirit of \cite{baier.hilgert.kaya.mourao.nunes:2023}, we want to study the limit of the family of K\"ahler quantizations $\mathcal{H}_{g+th}$ as $t\to +\infty.$

\begin{definition} The distributions $\delta_{\mu_{\mathrm{inv}}^{-1}(\lambda +\hat\rho)}$, $\lambda\in \hat U_K$, are defined by the unit mass measure supported on $\mu_{\mathrm{inv}}^{-1}(\lambda +\hat\rho)$ such that, for $f\in C_c(T^*(U/K)),$
\begin{equation}\label{deltas}
    \int_{T^*(U/K)} f \delta_{\mu_{\mathrm{inv}}^{-1}(\lambda +\hat\rho)} = \int_{U/K} f([x, \lambda+\hat\rho]) dx,
\end{equation}
where $dx$ is the normalized $U$-invariant measure on $U/K$.
\end{definition}

Let, for $\lambda\in \hat U_K , v^*\in V_\lambda^*,$
\begin{equation}\label{007}
\sigma^\infty_{\lambda, v^*} := c_0 (2\pi)^{l/2}i^l 
2^{-\# (\Sigma \cap \Phi_+)} P(\lambda+\hat\rho)^2 F_{\lambda, v_\lambda^K\otimes v^*} F_{\hat \rho, v^K_{\hat\rho}\otimes v_{\hat\rho}^K}\delta_{\mu_{\mathrm{inv}}^{-1}(\lambda +\hat\rho)} \otimes F^{-1}_{\hat \rho, v^K_{\hat\rho}\otimes v_{\hat\rho}^K}\tilde \Omega_\infty^{1/2},
\end{equation}
where
$P(\lambda+\hat \rho) := \Pi_{\alpha\in \Sigma\cap \Phi_+} \langle \alpha,\lambda+\hat\rho\rangle.$

\begin{theorem}\label{quantizationsconverge}
Let $g=0$ or $g\in \Conv^\infty_{\rm unif}(\uu^\ast)^{\Ad^\ast_U}$ and $h\in \Conv^\infty_{\rm unif}(\uu^\ast)^{\Ad^\ast_U}$. In the distributional sense, 
    $$
    \lim_{t\to +\infty} C_{t,h} \sigma^{g}_{\lambda,v^*} = \sigma^\infty_{\lambda, v^*},\,\, \lambda\in \hat U_K, v^*\in V_\lambda^*.
    $$
\end{theorem}

\begin{proof}
The proof follows the same calculations as in the proof of theorems 5.5 and 5.8 of \cite{baier.hilgert.kaya.mourao.nunes:2023}. 
\end{proof}

{}

\begin{corollary}
    The distributional sections $\sigma^\infty_{\lambda, v^*}, \, \lambda\in \hat U_K, v^*\in V_\lambda^*$ are $\mathcal{P}_\infty$-polarized.
\end{corollary}

\begin{proof}
    This is an analog of the proof of Corollary 5.7 of \cite{baier.hilgert.kaya.mourao.nunes:2023}: one can generate $\mathcal{P}_\infty$ by limits of Hamiltonian vector fields generating the K\"ahler polarizations and $\sigma^\infty_{\lambda,v^*}$ arises as a limit of K\"ahler polarized sections.
\end{proof}

In fact, these distributional sections obtained at infinite Mabuchi geodesic time comprise all of the polarized sections of $\mathcal{P}_\infty.$

\begin{theorem}The vector space of $\mathcal{P}_\infty$-polarized sections is given by the closure of 
$$
W_\infty := \bigoplus_{\lambda\in \hat U_K}\left\{ \sigma^\infty_{\lambda,v^*}\mid v^*\in V_\lambda^*,\, \lambda\in \hat U_K\right\}.
$$
in $\left(C_c(T^*(U/K))\right)^*\otimes \tilde \Omega^\frac12_\infty$.
\end{theorem}

{}

\begin{proof}
Half-form corrected $\mathcal{P}_\infty$-polarized sections can be written in the form 
$$
\sigma = s\otimes \Omega_\infty^\frac12,
$$
where, from Proposition \ref{polhalfform}, $\Omega_\infty^\frac12$ is already $\mathcal{P}_\infty$-polarized. Note that instead of $\Omega_\infty^\frac12$ one can use other polarized sections of the half-form bundle of $\mathcal{P}_\infty$, with different divisors, by multiplying $\Omega_\infty$ by ($\mathcal{P}_\infty$-polarized) factors of the form
$F_{\hat\rho, v_{\hat\rho^K}\otimes v_{\hat\rho}^K} F_{\hat\rho, v_{\hat\rho}\otimes z^*}^{-1}, z^*\in V_{\hat\rho}^*$ and by multiplying with the inverse factor on the left factor of the tensor product. The section $\sigma$ will then be $\mathcal{P}_\infty$-polarized iff $s$ is $\mathcal{P}_\infty$-polarized.
The prequantum connection on $T^*U$ is given by $\nabla = d + i\theta$  where the
connection form can be written as
$$
\theta = \sum_{j=1}^n \xi_j \tilde \omega_j,
$$
where $\left\{\tilde \omega_j\right\}_{j=1, \dots, n}$ is a basis of right-invariant 1-forms on $U$, corresponding to an orthonormal basis of $\uu$, pulled-back to $T^*U$ by the canonical projection $T^*U\to U$ (see Sections 2 and 4 of \cite{kirwin.mourao.nunes:2013}), and $\xi_j$ are the coordinates of $\xi$ in the corresponding dual basis of $\uu^*$. By symplectic reduction 
\cite{guillemin.sternberg:1982}, this induces the prequantum connection on the prequantum line bundle on $T^*(U/K)$, so that over a point $[u,\xi_+]\in T^*(U/K)$, with $\xi_+\in -i\aa^*_+$, the connection form reads
$$
\sum_{j=1}^l (\xi_+)_j \tilde \omega_j,
$$
where indices have been chosen such that $j=1, \dots ,l$ 
runs over a basis of $\aa= i(\tt\cap \ss).$ This connection form exactly matches the one that one obtains for symplectic toric manifolds, along the open dense subset diffeomorphic to $T^*((S^1)^l)$. Therefore, using Fourier decomposition of sections with respect to the 
action of $T_\mathrm{inv}$, the equations for covariant constancy along the real directions of 
$\mathcal{P}_\infty$, given precisely by the orbits of $T_\mathrm{inv}$, match the equations of covariant constancy along the real toric polarization of a symplectic toric manifold. 
From Proposition 3.1 in \cite{baier.florentino.mourao.nunes:2011}, the solutions are proportional to a Dirac delta distribution supported on the Bohr-Sommerfeld set, which in our case is given by the level sets $\mu_\mathrm{inv}^{-1}(\lambda+\hat \rho)$ for highest weights $\lambda\in \hat U_K,$ as mentioned in Section \ref{sec-fibering} and as can be explicitly checked from the previous formula for the connection form. In particular, there are no solutions proportional to derivatives of Dirac delta distributions. On the other hand, covariant constancy along the holomorphic directions of $\mathcal{P}_\infty$ just correspond to holomorphicity along the coadjoint orbits 
$\mathcal{O}_{\lambda+\hat\rho}$. But then all such holomorphic sections arise by taking the sections $\sigma^\infty_{\lambda,v^*}$ above, with $v^*\in V_\lambda^*.$ Indeed,
$V_{\lambda+\hat\rho}^*$ is an irreducible component in the tensor product 
$V_\lambda^*\otimes V_{\hat\rho}^*$.  
Thus $F_{\lambda+\hat\rho,v^K_{\lambda+\hat \rho}\otimes w^*}$, for $w\in V_{\lambda+\hat\rho}^*$, decomposes into sums of products of $F_{\lambda,v_\lambda^K\otimes v^*} F_{\hat\rho,v_{\hat\rho}^K\otimes z^*}$, $v^*\in V_\lambda^*, z^*\in V_{\hat\rho}^*.$
So, by multiplying and dividing 
$\Omega_\infty^\frac12$ by different 
$F_{\hat\rho, v^K_{\hat \rho}\otimes z^*}, \, z^*\in V_{\hat\rho}^*$, in (\ref{007}) we can describe all holomorphic sections of the Borel-Weil line bundle $L_{\lambda+\hat\rho}\to \mathcal{O}_{\lambda+\hat\rho}.$ (Note also that, as described in Proposition \ref{prop-holomorphicdirections}, the holomorphic directions of $\mathcal{P}_\infty$ along $\mu_\mathrm{inv}^{-1}(\lambda+\hat\rho)$ select as polarized functions elements of the ring generated by quotients of the form $F_{\lambda+\hat\rho,v^K_{\lambda+\hat \rho}\otimes w_1^*}F_{\lambda+\hat\rho,v^K_{\lambda+\hat \rho}\otimes w_2^*}^{-1}$, for $w_1,w_2\in V_{\lambda+\hat\rho}^*$, so that no other irreducible components of $V_\lambda^*\otimes V_{\hat\rho}^*$ contribute to give $\mathcal{P}_\infty$-polarized sections supported on that fiber of $\mu_\mathrm{inv}.$
\end{proof}
{}
Since $U$ itself can be described as a symmetric space of compact type, we obtain the 

\begin{corollary}The Hilbert space 
$$
\mathcal{H}_\mathrm{KW} = \widehat \bigoplus_{\lambda\in \hat U_K}{\left\{\sigma_{\lambda,A}^\infty \mid \lambda \in \hat U, A\in \mathrm{End}(V_\lambda)\right\}},
$$
as described in Section 5.3 of \cite{baier.hilgert.kaya.mourao.nunes:2023}, is the Hilbert space of 
polarized sections for the Kirwin-Wu polarization of $T^*U$.
\end{corollary}

{}

\subsection{The Hilbert space $\mathcal{H}_\infty$}
\label{subsection-hinfinity}

In this section, we will see that the inner products on the fibers of the quantum bundle induce naturally an inner product structure on $W_\infty$ so that one obtains the Hilbert space of 
$\mathcal{P}_\infty$-polarized sections, 
$\mathcal{H}_\infty$,  by taking the norm completion of $W_\infty$.

For the Hilbert space of the Schr\"odinger quantization $\mathcal{H}_\mathrm{Sch}=L^2(U/K)\otimes \sqrt{dx}$, we have the $U$-invariant inner 
product, 
$$
\langle \sigma_{\lambda,v^*}^0, \sigma_{\tilde \lambda, \tilde v^*}^0\rangle =
d_\lambda^{-1}\delta_{\lambda\tilde\lambda} \langle v^*, \tilde v^*\rangle_{V_\lambda^*}.
$$

Let us now study the evolution of the norms of $\sigma^{th}_{\lambda, v^*}$, $t>0$, for $h\in \Conv^\infty_{\rm unif}(\uu^\ast)^{\Ad^\ast_U}$ and $\lambda\in \hat U_K.$
From Theorem 3.3.25 in \cite{kaya:2015} we obtain that, for $g_\mathrm{std}=\frac12\vert\vert\xi\vert\vert^2$, $\vert\vert\mathcal{L}_{g_\mathrm{std}}^*\Omega^\frac12\vert\vert$ is an $\mathrm{Ad}^*_K$-invariant function 
$$
\vert\vert\mathcal{L}_{g_\mathrm{std}}^*\Omega^\frac12\vert\vert^2 ([x,\xi]) = \eta(\xi),
$$
where 
$$
\eta(\xi_+) := \Pi_{\alpha \in \Sigma\cap \Phi_+} \left(\frac{\mathrm{sinh}(2\alpha(\xi_+))}{2\alpha(\xi_+)}\right)^{\frac{m_\alpha}{2}},
$$
for $\xi_+\in -i\aa^*_+$.
For other K\"ahler structures determined by a symplectic potential 
$g\in\Conv^\infty_{\rm unif}(\uu^\ast)^{\Ad^\ast_U}$, we need to  pull-back by the composition of Legendre transforms $\LL_g \circ \LL_{g_{\mathrm{std}}}^{-1}$. This composition of Legendre transforms is not a symplectomorphism of $T^*(U/K)$ but it is straightforward to obtain the correcting factor for the pull-back of the Liouville measure (see, for instance, Lemma 2.4 in \cite{kirwin.mourao.nunes:2014}) to obtain for the contribution of the half-form
$$
\vert\vert\mathcal{L}_g^*\Omega^\frac12\vert\vert^2 ([x,\xi_+]) = \eta(d_{\xi_+}g).
\det({\mathrm{Hess_g}_{\vert_{\ss^*}}(\xi_+))^\frac12}.
$$
We will now study the behavior of the norms of the polarized sections 
$\sigma^{g_t}_{\lambda, v^*}$ as $t\to +\infty$ where, throughout this section
we will take
$$
g_t=th, \, t>0, h\in \Conv^\infty_{\rm unif}(\uu^\ast)^{\Ad^\ast_U},
$$
so that we are considering a Mabuchi geodesic ray that connects the Schr\"odinger (vertical) polarization to the polarization $\mathcal{P}_\infty$.
Since the half-form corrected inner product on $\mathcal{H}_{g_t}$ is $U$-invariant, we obtain that two different isotypical components $\mathcal{H}^\lambda_{g_t}, \mathcal{H}^{\tilde \lambda}_{g_t}$ are orthogonal for $\lambda \neq \tilde \lambda.$ Moreover, we have that on each isotypical component $\mathcal{H}^\lambda_g$ the inner product is fixed up to a multiplicative constant and it is determined, in particular, by the norm of the $K$-left-invariant vectors
$$
\sigma^{g_t}_{\lambda, (v_\lambda^K)^*}, \, \, \lambda\in \hat U_K
$$
 with $(v_\lambda^K)^*$ defined from $v_\lambda^K$ via the anti-linear bijection $V_\lambda\to V_\lambda^*$ explained in Appendix~\ref{appendix}.
For $\lambda\in \hat U_K$ let $v_\lambda^K$ be a unit norm spherical vector (for an $U$-invariant inner product on $V_\lambda$) and consider the Harish-Chandra $c$-function \cite[\S 4]{HPVinberg:2002},
$$
c(\lambda+\hat \rho) = \vert\langle v_\lambda, v_\lambda^K\rangle_{V_\lambda}\vert^2\not= 0.
$$
Note here that $v_\lambda$ and $v_\lambda^K$ were assumed to be unit vectors so that the expression for $c(\lambda+\hat\rho)$ is independent of the normalization of the inner product on $V_\lambda$.  

\begin{theorem}\label{notunitary}There is a non-zero 
constant $C_0$ (independent of $h$ and $\lambda$) such that
$$\lim_{t\to +\infty} \vert\vert \sigma^{g_t}_{\lambda, (v_\lambda^K)^*} 
\vert\vert_{\mathcal{H}_{g_t}}^2 = C_0\,  \hat P(\lambda+\hat\rho)^\frac12 c(\lambda+\hat\rho)  \,\vert\vert\sigma^{0}_{\lambda, (v_\lambda^K)^*} 
\vert\vert_{\mathcal{H}_{\mathrm{Sch}}}^2,$$
where
$\hat P(\xi_+) = \Pi_{\alpha\in \Sigma \cap \Phi_+} \alpha(\xi_+)^{m_\alpha}.$
\end{theorem}

{}

\begin{proof}From the expression for the half-form corrected inner product on the space of K\"ahler polarized sections (see, for example, Section 4 in \cite{kirwin.mourao.nunes:2013}), we need to compute
$$
\vert\vert \sigma^{g_t}_{\lambda, (v_\lambda^K)^*} 
\vert\vert_{\mathcal{H}_{g_t}}^2 = e^{-2th(\lambda +\hat \rho)}\cdot
$$
$$
\cdot \int_{T^*(U/K)} \overline{\mathrm{tr}(\pi_\lambda(x e^{itd_\xi h}v_\lambda^K\otimes (v_\lambda^K)^*))} \mathrm{tr}(\pi_\lambda(x e^{itd_\xi h}v_\lambda^K\otimes (v_\lambda^K)^*))  \cdot
$$
$$
\cdot e^{-t\kappa_{h}}\eta (td_\xi h) \det (\mathrm{Hess}_{th})_{\vert_{\ss^*}}^\frac12 \varepsilon,
$$
where $\varepsilon$ is the Liouville measure.
From, Theorem 3.3.41 in \cite{kaya:2015} (or Theorem A.4 in \cite{kaya:2020}), we obtain that, for some non-zero constant $a$, and $f$ any integrable function on 
$T^*(U/K)\cong U\times_K \ss^*$
$$
\int_{T^*(U/K)} f \varepsilon = a \int_{U\times \ss^*} (p^*f) dx d\xi,
$$
where $p:U\times \ss^* \to U\times_K \ss^*$ is the canonical projection, $dx$ is the normalized Haar measure on $U$ and $d\xi$ is the Lebesgue measure on $\ss^*.$ Applying this above and using the Weyl orthogonality relations
$$
\int_U \overline{\pi_\lambda (x)_{ij}} \pi_\lambda(x)_{kl} dx = 
d_\lambda^{-1} \delta_{ik} \delta_{jl},
$$
we obtain
$$
\vert\vert \sigma^{g_t}_{\lambda, (v_\lambda^K)^*} 
\vert\vert_{\mathcal{H}_{g_t}}^2 = a d_{\lambda}^{-1} e^{-2th(\lambda +\hat \rho)}\cdot
$$
$$
\cdot
\int_{\ss^*} \mathrm{tr}( \pi_\lambda(e^{2itd_\xi h}) 
v_\lambda^K\otimes (v_\lambda^K)^* ) e^{-t\kappa_{h}} \eta(td_\xi h) \det (\mathrm{Hess}_{th})_{\vert_{\ss^*}}^\frac12 d\xi,
$$
where we take a spherical vector $v^K_\lambda$ of norm one.
The integrand is $\mathrm{Ad}^*_K$-invariant. From Theorem I.5.17
in \cite{helgason:1984}, we obtain that for $f$ any integrable $\mathrm{Ad}^*_K$-invariant function on $\ss^*$
$$
\int_{\ss^*} f(\xi) d\xi = a' \int_{\aa^*_+} f(\xi_+) 
\hat P(\xi_+)d\xi_+
$$
for some non-zero constant $a'$.
Therefore,
$$
\vert\vert \sigma^{g_t}_{\lambda, (v_\lambda^K)^*} 
\vert\vert_{\mathcal{H}_{g_t}}^2 = 
$$
$$
= a^{''} d_{\lambda}^{-1} e^{-2th(\lambda +\hat \rho)}
\int_{\aa^*_+} \mathrm{tr}( \pi_\lambda(e^{2itd_{\xi_+} h}) 
v_\lambda^K\otimes (v_\lambda^K)^* ) e^{-t\kappa_{h}} \eta(td_{\xi_+} h) \det (\mathrm{Hess}_{th})_{\vert_{\ss^*}}^\frac12 \hat P (\xi_+) d\xi_+,
$$
where $a^{''}$ is a non-zero constant.
Therefore, the leading term in the expression for the norm as $t\to +\infty$ is
$$
\lim_{t\to +\infty} \vert\vert \sigma^{g_t}_{\lambda, (v_\lambda^K)^*} 
\vert\vert_{\mathcal{H}_{g_t}}^2 = a^{''} d_{\lambda}^{-1} c(\lambda+\hat\rho)\cdot
$$
$$
\lim_{t\to +\infty} t^{d/2} e^{-2th(\lambda +\hat \rho)}
\int_{\aa^*_+} e^{-2t\langle\lambda,d_{\xi_+} h\rangle} 
 e^{-t\kappa_{h}} \eta(td_{\xi_+} h) \det (\mathrm{Hess}_{h})_{\vert_{\ss^*}}^\frac12 \hat P (\xi_+) d\xi_+,
$$
where $c(\lambda+\hat\rho)$ is the Harish-Chandra $c$-function and $d=\dim \ss^*$.
From (\ref{e-rest-root-space-dec}) and the proof of Proposition \ref{prop-leftinvcanonical}, we obtain that 
$$
\mathrm{Hess}_h(\xi_+) (E_\alpha - E_{\alpha^\sigma}) = 
(2\alpha(\xi_+))^{-1} \mathrm{ad}^*_{(d_{\xi_+}h)^*} (E_\alpha + E_{\alpha^\sigma}),
$$
so that along the subspace 
$$
(\aa^*)^\perp = \bigoplus_{\alpha\in (\Phi\setminus\Phi_0)\cap \Phi_+} (E_\alpha - E_{\alpha^\sigma})\subset \ss^*
$$
$\mathrm{Hess}_{h}(\xi_+)$ is diagonal with entries
$$
\frac{\alpha (d_{\xi_+}h)}{\alpha(\xi_+)},
$$
which is analogous to the case of $T^*U$ (see the proof of Lemma 4.14 in \cite{baier.hilgert.kaya.mourao.nunes:2023}). Therefore
$$
\det \mathrm{Hess}_{h} (\xi_+)_{\vert_{\ss^*}} = \frac{\hat P((d_{\xi_+}h)^*)}{\hat P(\xi_+)} \det \mathrm{Hess}_{h} (\xi_+)_{\vert_{\aa^*}}.
$$
{}
Therefore, the leading term for the norm as $t\to +\infty$ is
$$
\lim_{t\to +\infty} \vert\vert \sigma^{g_t}_{\lambda, (v_\lambda^K)^*} 
\vert\vert_{\mathcal{H}_{g_t}}^2 = a^{''} d_{\lambda}^{-1} c(\lambda+\hat\rho) \cdot
$$
$$
\cdot \lim_{t\to +\infty}  t^{d/2}e^{-2th(\lambda +\hat \rho)}
\int_{\aa^*_+} e^{-2t\langle\lambda,d_{\xi_+} h\rangle} 
 e^{-t\kappa_{h}} \eta(td_{\xi_+} h) \det (\mathrm{Hess}_{h})_{\vert_{\aa^*}}^\frac12 (\hat P (\xi_+))^\frac12 (\hat P ((d_{\xi_+}h)^*))^\frac12 d\xi_+.
$$
where in leading order
$$
\eta (td_{\xi_+}h) \sim \mathrm{const.}\, t^{-\# (\Sigma\cap \Phi_+)}e^{2t\hat\rho(d_{\xi_+}h)} P((d_{\xi_+}h)^*)^{-\frac12}.
$$

We obtain,
$$
\lim_{t\to +\infty} \vert\vert \sigma^{g_t}_{\lambda, (v_\lambda^K)^*} 
\vert\vert_{\mathcal{H}_{g_t}}^2 = a^{'''} d_{\lambda}^{-1} c(\lambda+\hat\rho) \cdot
$$
$$
\cdot\lim_{t\to +\infty}  t^{l/2}e^{-2th(\lambda +\hat \rho)}
\int_{\aa^*_+} e^{-2t\langle(\lambda +\hat\rho),d_{\xi_+} h\rangle} 
 e^{-t\kappa_{h}}  \det (\mathrm{Hess}_{h})_{\vert_{\aa^*}}^\frac12 (\hat P (\xi_+))^\frac12  d\xi_+,
$$
where $a^{'''}$ is a non-zero constant.
Just as in the proof of Theorem 4.1 in \cite{kirwin.mourao.nunes:2014}
or in the proof of Lemma 5.4 in \cite{baier.hilgert.kaya.mourao.nunes:2023}, the exponentials, the power $t^{l/2}$ and the determinant of the Hessian of $h$ along $\aa^*$ produce in the limit $t\to +\infty$, up to a constant,
a delta function $\delta(\lambda+\hat\rho)$ so that 
$$
\lim_{t\to +\infty} \vert\vert \sigma^{g_t}_{\lambda, (v_\lambda^K)^*} 
\vert\vert_{\mathcal{H}_{g_t}}^2 = a^{''''} d_{\lambda}^{-1} \hat P(\lambda+\hat\rho)^\frac12 c(\lambda+\hat\rho),
$$
for a non-zero constant $a^{''''}$.
{}
\end{proof}

\begin{remark}
    The factor of $c(\lambda+\hat \rho)$ in the statement of Theorem \ref{notunitary} is consistent with the results of Stenzel in \cite{stenzel:1999}. By identifying the fibers of $T^*(U/K)$ with the non-compact dual symmetric space to $U/K$, he considers a natural measure on $T^*(U/K)$ for which the time-$t$ CST, 
    for the case $h=\frac12 \vert\vert\xi\vert\vert^2$, becomes as unitary transform from $L^2(U/K)$ to a space of holomorphic functions on $U_\CC/K_\CC$. This measure involves the time-$2t$ heat kernel on the non-compact dual symmetric space (see, for instance, \cite{anker.ostellari.2004}) which carries a factor of 
    $\vert c(\mu)\vert^{-2}$ (from the Plancherel formula for non-compact symmetric spaces) and an integration along $\mu\in \aa^*$. From 
    Theorem 3 in \cite{stenzel:1999}, since this CST is unitary with respect to this choice of measure, one can evaluate the norm of 
    $\mathrm{tr} (\pi_\lambda(x e^{i\xi_+}) v_\lambda^K\otimes v^* ) $ asymptotically as $t\to +\infty.$ The spherical function $\varphi_\mu$ which features in the heat kernel will give a factor of $c(\mu)$ in the asymptotic limit (see Section 2 of \cite{Helgason:1964} or Section 4 of \cite{HPVinberg:2002}). The integration along $\xi_+$ then produces Dirac delta functions supported on the points $(\nu+\hat\rho)$, where $\nu$ runs over the weights of $V_\lambda$. In the asymptotic limit only the highest weight survives and the remaining factor of $\vert c\vert^{-1}$ localizes on $(\lambda+\hat\rho)$, consistently with Theorem \ref{notunitary}.
\end{remark}
{}

Using Theorems \ref{quantizationsconverge} and \ref{notunitary}, we define an inner product on
$W_\infty$ induced from the asymptotic limit of the inner products on the Hilbert spaces for half-form corrected K\"ahler quantizations $\mathcal{H}_\mathrm{th}$ along {\it any} geodesic path of $U$-invariant K\"ahler structures $th, t>0,$ that is by declaring that the generators
$$
\left(\frac{\hat P (\lambda+\hat\rho)^\frac12 c(\lambda+\rho)}{d_\lambda}\right)^{-\frac12} \sigma^\infty_{\lambda, e_j^*}, \lambda\in \hat U_K, 
$$
give an orthonormal set, where $e_j, j=1, \dots, d_\lambda$ is an orthonormal basis for $V_\lambda$.
Taking the norm completion we obtain 
the Hilbert space 
$$
\mathcal{H}_\infty := \widehat\bigoplus_{\lambda\in \hat U_K}{\left\{\sigma^\infty_{\lambda,v^*}\mid  \lambda\in \hat U_K, v^*\in V_\lambda^* \right\}}
$$
of $\mathcal{P}_\infty$-polarized sections.

Since the inner products on $U$-isotypical components in $\mathcal{H}_\mathrm{Sch}$ and in 
$\mathcal{H}_\infty$ are related by $\lambda$-dependent constants we obtain, as a corollary, 

\begin{theorem}\label{thm_asympnot}
The maps $C_{t,h}: \mathcal{H}_{\mathrm{Sch}} \to \mathcal{H}_{th}$ are not asymptotically unitary up to a ($h$-independent) constant as $t\to +\infty$ for any $h\in  \Conv^\infty_{\rm unif}(\uu^\ast)^{\Ad^\ast_U}.$
\end{theorem}
{}

Note, however, that for Mabuchi rays generated by quadratic Hamiltonians, the Laplace approximation used in the proof of Theorem \ref{notunitary} is exact so that, for strictly positive time, the corresponding  gCST are unitary, that is we have

\begin{proposition}Let $h\in \Conv^\infty_{\rm unif}(\uu^\ast)^{\Ad^\ast_U}$ be quadratic and $s>0$. Then, the $U$-equivariant maps 
$$
C_{t,h}: \mathcal{H}_{sh} \to \mathcal{H}_{(s+t)h}
$$
are unitary for $t> -s$.    
\end{proposition}

Thus, the inner product structure on $\mathcal{H}_\infty$ can be obtained through the continuous family of unitary gCST maps for the Mabuchi geodesic going through quadratic Hamiltonians. For geodesics generated by non-quadratic Hamiltonians one obtains only asymptotic unitarity of the 
gCST maps $C_{t,h}$. Moreover, the $U$-invariant inner product structure for 
$\mathcal{H}_\mathrm{Sch}$ is related to the inner products for the Hilbert spaces of the K\"ahler polarizations and, hence also to the inner product on $\mathcal{H}_\infty$, through 
representation-dependent factors in the $U$-isotypical decompositions.

\begin{remark} 
     In the case of $T ^*U$ \cite{baier.hilgert.kaya.mourao.nunes:2023}, the inner product for the limit Kirwin-Wu polarization is also induced asymptotically by taking the inner product along Mabuchy rays. By contrast, however, one can take the rays to begin at the vertical polarization, that is, in the case of $T^*U$, the inner products for $\mathcal{H}_\mathrm{Sch}$, for the K\"ahler polarizations $\mathcal{H}_g,\, g\in \Conv^\infty_{\rm unif}(\uu^\ast)^{\Ad^\ast_U}$, and for $\mathcal{H}_\mathrm{KW}$ are related $t,h$-dependent  overall factors which do not vary along the $U$-isotypical components. (In the case of Mabuchi rays of quadratic Hamiltonians, the gCST provides a continuous family of unitary maps from $\mathcal{H}_\mathrm{Sch}$ to $\mathcal{H}_\mathrm{KW}$.)
\end{remark}

\subsection{The extended quantum bundle}
\label{subsection_flatconn}

By including the fibers corresponding to the vertical polarization and to $\mathcal{P}_\infty$, 
 one then obtains an extended bundle of quantization Hilbert spaces 
 $$
 \mathcal{H}\to \Conv^\infty_{\rm unif}(\uu^\ast)^{\Ad^\ast_U} \cup \left\{0,\infty\right\}
 $$ 
 with fibers $\mathcal{H}_g$ over $g\in \Conv^\infty_{\rm unif}(\uu^\ast)^{\Ad^\ast_U}$, 
 $\mathcal{H}_0=\mathcal{H}_\mathrm{Sch}$ over $0$ and $\mathcal{H}_\infty$ over $\infty.$ 
As described in the previous sections, $\mathcal H$  comes equipped with a global trivializing frame
$$
\left\{ \sigma_{\lambda, v^*}^{g}\mid \lambda\in \hat U_K, v^*\in V_\lambda^*, g\in \left\{0\right\}\cup \Conv^\infty_{\rm unif}(\uu^\ast)^{\Ad^\ast_U} \cup \left\{\infty \right\}\right\},
$$
which is parallel with respect to the flat connection $\nabla^Q$ extended to $\mathcal{H}$.

The gCST maps $C_{t,h}, t\geq 0, h\in \Conv^\infty_{\rm unif}(\uu^\ast)^{\Ad^\ast_U},$ give $U$-equivariant isomorphisms between the fibers of $\mathcal{H}$ and define the parallel transport operators of $\nabla^Q$. If one excludes the fiber corresponding to the vertical polarization, the connection $\nabla^Q$ is unitary along the Mabuchi rays going through quadratic Hamiltonians while for other rays one obtains asymptotic unitarity at infinite geodesic time.

{}
\subsection{Comparing $C_{\infty,h}$ with the unitary Fourier transform}
\label{subsection_fouriertransform}

As in the case of $T^*U$, parallel transport in the extended quantum bundle along the geodesic paths $th, h\in \Conv^\infty_{\rm unif}(\uu^\ast)^{\Ad^\ast_U}, t\geq 0,$
gives an $U$-equivariant isomorphism defined by the linear extension of 
$$
L^2(U/K,dx) \cong \mathcal{H}_{\mathrm{Sch}} \ni 
 \tr (\pi_\lambda (x) v_\lambda^K\otimes v^*) \otimes \sqrt{dx} \mapsto \sigma^\infty_{\lambda, v^*} \in \mathcal{H}_\infty,\,\, \lambda\in \hat U_K, v^*\in V_\lambda^*.
$$

Clearly we have an $U$-equivariant isomorphism
$$
\Phi_\mathrm{GQ}:\mathcal{H}_\infty \cong \widehat{\bigoplus_{\lambda\in \hat U_K}}  V_\lambda^*
$$
with
$$
\Phi_\mathrm{GQ} (\sigma^\infty_{\lambda, v^*}) = v^*.
$$

From above, this defines a unitary (up to a constant) vector valued Fourier transform with respect to the inner product on 
$\widehat{\bigoplus_{\lambda\in \hat U_K}}  V_\lambda^*$ such that, for a basis  of $V_\lambda$, $\left\{e_j, j=1, \dots, d_\lambda\right\}$, orthonormal for a choice of $U$-invariant inner product,
$$
\vert\vert e_j^*\vert\vert_{\mathrm{GQ}}= \left(\frac{\hat P (\lambda+\hat\rho)^\frac12 c(\lambda+\rho)}{d_\lambda}\right)^{-\frac12}.
$$

The standard unitary vector-valued Fourier transform for compact symmetric spaces (see \cite{helgason:1984} and Appendix A) 
$$
\mathcal{F}: L^2(U/K) \to \widehat{\bigoplus_{\lambda\in \hat U_K}} V_\lambda\cong \widehat{\bigoplus_{\lambda\in \hat U_K}} V_\lambda^*
$$
can be defined by setting 
$$
L^2(U/K) \ni \sqrt{d_\lambda} \tr (\pi_\lambda (x) v_\lambda^K\otimes v^*)
\mapsto v \in V_\lambda.
$$
$\mathcal{F}$ is a unitary isomorphism of Hilbert spaces. 

Therefore, unlike what happens for the case of $T^*U$, the vector-valued Fourier transform induced by geometric quantization along the Mabuchi rays of $U$-invariant K\"ahler structures that we describe in this paper is not identical to the standard Fourier transform. That is, if one considers the natural inner products making all the arrows in the following diagram unitary maps,
\begin{equation*}
\begin{tikzcd} 
	L^2(U/K) \ar[r]\ar[d,"\mathcal F"]	& \mathcal H_\mathrm{Sch}\ar[dd]\\
	\widehat{\bigoplus}_{\lambda\in\hat U_K} V_\lambda\ar[d]	 &\\
	\widehat{\bigoplus}_{\lambda\in\hat U_K} V_\lambda^*&
    \mathcal H_\infty\ar[l,"\Phi_\mathrm{GQ}"]
\end{tikzcd}
\end{equation*}
then, the diagram commutes only up to multiplication by the representation dependent factor $\hat P(\lambda+\hat\rho)^{-\frac12} c(\lambda + \hat\rho)^{-1}$, along the isotypical component for $\lambda\in \hat U_K$, which is necessary to make $\Phi_\mathrm{GQ}$ unitary with respect to the inner product structure on $\mathcal{H}_\infty$ which is, as we described, induced by taking a limit of the half-form corrected inner products for the K\"ahler quantizations along Mabuchi geodesics connecting the vertical polarizations and $\mathcal{P}_\infty.$

{}
\subsection{Quantum-geometric interpretation of $\mathcal{P}_\infty$}
\label{subsection_interpretation}

The limit polarization $\mathcal{P}_\infty$ has an important quantum-geometric interpretation, in line with the general program outlined in Section 7 of \cite{baier.hilgert.kaya.mourao.nunes:2023}.
The Bohr-Sommerfeld set of $\mathcal{P}_\infty$, as we have seen, is given by 
$$
\bigcup_{\lambda\in \hat U_K} \mu_\mathrm{inv}^{-1}(\lambda+\hat\rho),
$$
where these level sets were called ``spectral manifolds" 
in \cite{baier.hilgert.kaya.mourao.nunes:2023}. In the limit polarization $\mathcal{P}_\infty$, the quantization of the coordinate components of $\mu_\mathrm{inv}$ is indeed given just by multiplication operators with spectrum determined by evaluation at $\lambda+\hat\rho,\, \lambda\in \hat U_K.$

From Section \ref{sec-fibering} and Proposition \ref{prop-holomorphicdirections}, the coresponding symplectic reductions for the Hamiltonian action of $T_\mathrm{inv}$ on $T^*(U/K)_\mathrm{reg}$ give coadjoint orbits 
$$
\mu_{\mathrm{inv}}^{-1} (\lambda+\hat\rho) / T_\mathrm{inv} \cong 
\mathcal{O}_{\lambda+\hat\rho}.
$$

Let 
$$
\mathcal{H}_\infty^\lambda := \oplus_{v^*\in V_\lambda^*} \langle \sigma^\infty_{\lambda, v^*}\rangle_\CC \subset \mathcal{H}_\infty, \, \, \lambda\in \hat U_K .
$$
{}
From Proposition \ref{prop-holomorphicdirections} and (\ref{007}), for $\lambda\in \hat U_K,$ we obtain a natural $U$-equivariant linear isomorphism 
$$
H^0(\mathcal{O}_\lambda, L_\lambda) \cong \mathcal{H}_\infty^\lambda,
$$
identifying the quantization of the (integral) coisotropic reductions of $\mathcal{P}_\infty$ with subspaces of $\mathcal{H}_\infty$, so that 
$$
\mathcal{H}_\infty = \overline{\oplus_{\lambda\in \hat U_K} \mathcal{H}_\infty^\lambda} \cong 
\overline{\oplus_{\lambda\in \mu_\mathrm{inv}(T^*(U/K)_\mathrm{reg})\cap \Lambda_+^K} H^0(\mathcal{O}_\lambda, L_\lambda)}.
$$

Thus, as described in Section 7 of \cite{baier.hilgert.kaya.mourao.nunes:2023}, the Hilbert space for the quantization in the limit polarization $\mathcal{P}_\infty$ ``decomposes" as a sum of the holomorphic quantizations of the symplectic reductions, with respect to the Hamiltonian action of $T_\mathrm{inv}$, of the spectral manifolds given by the components of the Bohr-Sommerfeld set for $\mathcal{P}_\infty.$
{}

\appendix

\section{The Fourier transform}\label{appendix}

\cite[Thm.~V.4.3]{helgason:1984} shows that one has a unitary vector valued Fourier transform 
\[\mathcal F: L^2(U/K)\to \widehat{\bigoplus_{\lambda\in\hat U_K}} V_\lambda,\]
defined by 
\[\mathcal F f=\sum_{\lambda\in\hat U_K} d_\lambda\int_U \overline{\chi_\lambda}(u) L_uf\, du,\]
where $L_u$ is the left regular representation of $U$ on $L^2(U/K)$. Here we use the isometric embeddings 
\[V_\lambda\to L^2(U/K),\quad v\mapsto \sqrt{d_\lambda}f_v\]
where
\[f_{\lambda,v}(u):=\langle v\mid \pi_\lambda(u) v_\lambda^K\rangle_{V_\lambda}\]
and $\langle\cdot\mid \cdot\rangle_{V_\lambda}$ denotes the inner product on $V_\lambda$. Note that 
\[\mathcal Ff_{\lambda,v}=f_{\lambda,v}\]

Let $V_\lambda^*$ be the complex dual of $V_\lambda$. The characterization of the spherical dual $\hat U_K$ given in \cite[Thm. V.4.1]{helgason:1984} shows that the contragredient representation $(\tilde \pi_\lambda, V_\lambda^*)$ of $(\pi_\lambda,V_\lambda)$ is spherical as well. Recall that
\[\langle\tilde \pi_\lambda(u)\nu,v\rangle=\langle\nu,\tilde \pi_\lambda(u^{-1})v\rangle\]
for $U\in U, v\in V_\lambda, \nu\in V_\lambda^*$ and $\langle\cdot,\cdot\rangle$ denoting the natural pairing of $V_\lambda^*$ and $V_\lambda$. Writing $\langle w\mid v\rangle_{V_\lambda}=\langle v^*,w\rangle$ then defines an anti-linear bijection 
\[V_\lambda\to V_\lambda^*,\quad v\mapsto v^*\] 
which is $U$-equivariant with respect to $\pi_\lambda$ and $\tilde\pi_\lambda$. We equip $V_\lambda^*$ with the inner product making this map an isometry, i.e.
\[\langle v^*\mid w^*\rangle_{V_\lambda^*}=\langle w\mid v\rangle_{V_\lambda}=\langle v^*, w\rangle.\]
Collecting these maps for all $\lambda\in \hat U_K$ we obtain an anti-linear bijective isometry 
\[\widehat{\bigoplus_{\lambda\in\hat U_K}} V_\lambda\longrightarrow\widehat{\bigoplus_{\lambda\in\hat U_K}} V_\lambda^*.\]
In view of 
\[\langle v\mid \pi_\lambda(u) v_\lambda^K\rangle_{V_\lambda}
=\overline{\langle \pi_\lambda(u) v_\lambda^K\mid  v\rangle_{V_\lambda}}
= \overline{\langle  v^* , \pi_\lambda(u) v_\lambda^K\rangle } \]
we have isometric embeddings  
\[V_\lambda^*\to L^2(U/K),\quad v^*\mapsto \sqrt{d_\lambda}f_{\lambda,v^*}\]
where
\[f_{\lambda,v^*}(u):=\langle v^*, \pi_\lambda(u) v_\lambda^K\rangle\]

\bigskip
\bigskip
\bigskip
{\bf Acknowledgements:} TB, JM and JN were supported by the projects UIDB/04459/2020 and UIDP/04459/2020. The research of ACF was financed by Portuguese Funds through FCT (Fundação para a Ciência e a Tecnologia, I.P.) within the Projects UIDB/00013/2020 and UIDP/00013/2020. JH was partially funded by the Deutsche Forschungsgemeinschaft (DFG, German Research Foundation) – Project-ID 491392403 – TRR 358.

\begingroup
\sloppy
%\raggedright
\printbibliography
\endgroup

\bigskip

%\bibliography{bhkmn_biblio}
%\bibliographystyle{alpha}

\end{document}